\documentclass{amsart}
\usepackage{color}

\newtheorem{them}{Theorem}[section]
\newtheorem{lem}{Lemma}[section]

\newtheorem{cor}{Corollary}[section]

\theoremstyle{remark}
\newtheorem{rem}{Remark}[section]

\numberwithin{equation}{section}

\begin{document}
\title[On product affine hyperspheres in $\mathbb{R}^{n+1}$]
{On product affine hyperspheres in $\mathbb{R}^{n+1}$}

\author[X. Cheng, Z. Hu, M. Moruz and L. Vrancken]
{Xiuxiu Cheng, Zejun Hu, Marilena Moruz and Luc Vrancken}


\thanks{The first two authors were supported by NSF of China, Grant Number
11771404. The third author is a postdoctoral fellow
of FWO-Flanders, Belgium.}



\begin{abstract}
In this paper, we study locally strongly convex affine hyperspheres
in the unimodular affine space $\mathbb{R}^{n+1}$ which, as
Riemannian manifolds, are locally isometric to the Riemannian
product of two Riemannian manifolds both possessing constant
sectional curvatures. As the main result, a complete classification
of such affine hyperspheres is established. Moreover, as direct
consequences, affine hyperspheres of dimensions 3 and 4 with
parallel Ricci tensor are also classified.
\end{abstract}

\maketitle

\vskip 20pt
\section{Introduction}\label{sect:1}

In this paper, we study locally strongly convex affine hypersurfaces
in the unimodular affine space $\mathbb{R}^{n+1}$. It is well known that on a nondegenerate
affine hypersurface $M^n$ in $\mathbb{R}^{n+1}$ there exists a
canonical transversal vector field $\xi$ which is called the affine
normal vector field. If all the affine normal lines of $M^n$ pass
through a fixed point (resp. if all the affine normals are
parallel), $M^n$ is called a proper (resp. improper) affine
hypersphere. The second fundamental form $h$ associated with the
affine normal vector field is called the (Blaschke) affine metric. As we
consider only locally strongly convex affine hypersurfaces, the
affine metric $h$ is assumed to be positive definite, and in such
situation, the proper affine hyperspheres are divided into two
classes, i.e., the elliptic affine hyperspheres and the hyperbolic
ones.

The affine hyperspheres form a very important class of affine
hypersurfaces. From the global point of view that the affine metric
$h$ is complete, the improper (also called parabolic) affine hypersphere
has to be the elliptic paraboloid, whereas the elliptic affine
hypersphere has to be the ellipsoid. However, the class of locally
strongly convex hyperbolic affine hyperspheres is very large and
have been widely studied, see amongst others the works of
\cite{C,CY,G,Li1,Li2,Li3,Sa} and also the recent monograph \cite{LSZH},
or the survey paper \cite{Lo}. Indeed, even assuming global
conditions, the class of hyperbolic affine hyperspheres is
surprisingly large, and one is still far from having a complete
geometric understanding of them for all dimensions.

On the other hand, affine hyperspheres with constant sectional
curvature are classified in \cite{LP} and \cite{VLS} (see also
\cite{Si,V} for the general non-degenerate case), whereas in
\cite{HLV3} it was further shown that all locally strongly convex
Einstein affine hyperspheres in $\mathbb{R}^5$ are of constant
sectional curvature. Contrary to the result of \cite{HLV3}, the
cases for locally strongly convex Einstein affine hyperspheres in
$\mathbb{R}^{n+1}$ with $n\ge5$ are different, and there exist
Einstein affine hyperspheres which are not of constant sectional
curvatures; actually, such examples occur for the standard
embeddings of the noncompact symmetric spaces
$\mathrm{E}_{6(-26)}/\mathrm{F}_4$, and
$\mathrm{SL}(m,\mathbb{R})/\mathrm{SO}(m)$,
$\mathrm{SL}(m,\mathbb{C})/\mathrm{SU}(m)$,
$\mathrm{SU}^*(2m)/\mathrm{Sp}(m)$ for each $m\ge3$ (cf.
\cite{BD,HLV2} and \cite{CH,CHLL}). However, at present the complete
classification of locally strongly convex Einstein affine
hyperspheres in $\mathbb{R}^{n+1}$ is still an interesting and open
problem.

In order to get further knowledge of the affine hyperspheres, the
above mentioned facts motivate us to consider the following natural
and interesting problem:

\vskip 1mm

{\it Classify all locally strongly convex affine hyperspheres which
are locally isometric to the product $M_1^{n_1}(c_1)\times
M_2^{n_2}(c_2)$, such that $n_1+n_2=n$ and $M_i^{n_i}(c_i)$ is an
$n_i$-dimensional Riemannian manifold with constant sectional
curvature $c_i$ for $i=1,2$.}

\vskip 1mm

To consider this problem, we are sufficient to assume that $n\ge3$.
As the results of this paper, we have solved the above problem. More
precisely, we have proved the following theorems.

\begin{them}\label{thm:1.1}
Let $x:M^{n}\rightarrow\mathbb{R}^{n+1}$ be a locally strongly
convex affine hypersphere. If $(M^n,h)$ is locally isometric to the
Riemannian product $M_1^{n_1}(c_1)\times M_2^{n_2}(c_2)$ for
$n_1\ge2$ and $n_2\ge2$, such that $n_1+n_2=n$ and $M_i^{n_i}(c_i)$
is an $n_i$-dimensional Riemannian manifold with constant sectional
curvature $c_i$ for $i=1,2$. Then we have $c_1c_2=0$, and one of the
following cases occurs:
\begin{enumerate}
\item[(i)] $c_1=c_2=0$ and $x:M^{n}\rightarrow\mathbb{R}^{n+1}$ is locally affinely equivalent to either the 
paraboloid $x_{n+1}=\tfrac12[(x_1)^2+\cdots+(x_n)^2]$ or $Q(1,n):\
x_1x_2\cdots x_{n+1}=1$;

\item[(ii)]$c_1c_2=0$ and $c_1^2+c_2^2\neq0$, assuming that $c_1=0$ and $c_2\neq0$, then
$c_2<0$, $x:M^{n}\rightarrow\mathbb{R}^{n+1}$ is locally affinely
equivalent to the Calabi composition
$$
(x_1\cdots x_{n_1})^2(x_{n+1}^2-x_{n_1+1}^2-\cdots-x_{n}^2)^{n_2+1}=1,
$$
\end{enumerate}
where $(x_1,\ldots, x_{n+1})$ are the standard coordinates of
$\mathbb{R}^{n+1}$.
\end{them}

\begin{them}\label{thm:1.2}
Let $x:M^n\rightarrow\mathbb{R}^{n+1}\ (n\ge3)$ be a locally
strongly convex affine hypersphere. If $(M^n,h)$ is locally
isometric to a Riemannian product $I\times\tilde M^{n-1}(c)$, with
$I\subset\mathbb R$ and $\tilde M^{n-1}(c)$ an $(n-1)$-dimensional
Riemannian manifold with constant sectional curvature $c\not=0$.
Then we have $c<0$, and $x:M^n\rightarrow\mathbb{R}^{n+1}$ is
locally affinely equivalent to the Calabi composition
\begin{equation*}
x_1^2(x_{n+1}^2-x_{2}^2-\cdots-x_{n}^2)^{n}=1,
\end{equation*}
where $(x_1,\ldots, x_{n+1})$ are the standard coordinates of
$\mathbb{R}^{n+1}$.
\end{them}

As direct consequences of these theorems, we further have the
following results.

\begin{cor}\label{cor:1.1}
Let $x:M^3\rightarrow\mathbb{R}^{4}$ be a locally strongly convex
affine hypersphere with parallel Ricci tensor. Then either $M^3$ is
an open part of a locally strongly convex hyperquadric, or
$x:M^3\rightarrow\mathbb{R}^{4}$ is locally affinely equivalent to
one of the following two hypersurfaces:
\begin{enumerate}
\item[(i)] $x_1x_2x_3x_4=1$,

\vskip 1mm

\item[(ii)] $x_1^2(x_4^2-x_2^2-x_3^2)^3=1$,
\end{enumerate}
where $(x_1,x_2,x_3, x_4)$ are the standard coordinates of
$\mathbb{R}^4$.
\end{cor}

\begin{cor}\label{cor:1.2}
Let $x:M^4\rightarrow\mathbb{R}^{5}$ be a locally strongly convex
affine hypersphere with parallel Ricci tensor. Then either $M^4$ is
an open part of a locally strongly convex hyperquadric, or
$x:M^4\rightarrow\mathbb{R}^{5}$ is locally affinely equivalent to
one of the following hypersurfaces:
\begin{enumerate}
\item[(i)] $x_1x_2x_3x_4x_5=1$,

\vskip 1mm

\item[(ii)] $(x_1x_2)^2(x_5^2-x_3^2-x_4^2)^3=1$,

\vskip 1mm

\item[(ii)] $x_1^2(x_5^2-x_2^2-x_3^2-x_4^2)^4=1$,
\end{enumerate}
where $(x_1,x_2,x_3, x_4, x_5)$ are the standard coordinates of
$\mathbb{R}^5$.
\end{cor}

\begin{rem}\label{rem:1.1}
The above corollaries and the main results of \cite{DV1} and
\cite{DVY} imply that for locally strongly convex affine
hyperspheres in both $\mathbb{R}^4$ and $\mathbb{R}^5$, the parallelism
of the intrinsic invariant {\it Ricci tensor} and that of the
extrinsic invariant {\it cubic form} are actually equivalent.
\end{rem}

The paper is arranged as follows: In section 2, we fix notations and
briefly recall the local theory of equiaffine hypersurfaces. In
section 3, the most technical parts of this paper are given and we
prove the crucial lemmas which imply the existence of canonical
local frame so that the difference tensor can be sufficiently
determined. Finally, in section 4 we complete the proof of the
preceding theorems and corollaries.

\section{Preliminaries}\label{sect:2}

In this section, we briefly recall the local theory of equiaffine
hypersurfaces. For more details, we refer to the monographs \cite{LSZH, NS94}.

Let $\mathbb{R}^{n+1}$ be the standard $(n+1)$-dimensional real unimodular
affine space that is equipped with its usual flat connection $D$ and a
parallel volume form given by the determinant. Let $x:M^n\rightarrow
\mathbb{R}^{n+1}$ be a locally strongly convex hypersurface with
affine normal $\xi$. Then, for any vector fields $X$ and $Y$ on
$M^n$, we have
\begin{align}
&D_Xx_*(Y)=x_*(\nabla_XY)+h(X,Y)\xi,\label{eqn:2.1}\\[1mm]
&D_X\xi=-x_*(SX),\label{eqn:2.2}
\end{align}
where $\nabla, S$ and $h$ are the induced affine connection, the
affine shape operator and the affine metric, respectively. It is
well known that $M^n$ is an affine hypersphere if and only if
$S=H\,{\rm id}$ with $H$ being a constant; moreover,
$x:M^n\rightarrow \mathbb{R}^{n+1}$ is a proper (resp. improper)
affine hypersphere if and only if $H\neq 0$ (resp. $H=0$).

Let $\hat{\nabla}$ denote the Levi-Civita connection of the affine
metric $h$. The difference tensor $K$ is defined by
$K(X,Y):=K_XY:=\nabla_XY-\hat\nabla_XY$; it is symmetric as both
connections are torsion free. Moreover, $h(K(X,Y),Z)$ is a totally
symmetric cubic form. For affine hyperspheres with affine shape
operator $S=H \,{\rm id}$, the Riemannian curvature tensor $\hat{R}$
of the affine metric and the difference tensor $K$ satisfy the
following fundamental equations of Gauss and Codazzi:
\begin{equation}\label{eqn:2.3}
\hat{R}(X,Y)Z=H\big[h(Y,Z)X-h(X,Z)Y\big]-[K_X,K_Y]Z,
\end{equation}
\vskip -5mm
\begin{equation}\label{eqn:2.4}
(\hat{\nabla}_XK)(Y,Z)=(\hat{\nabla}_YK)(X,Z).
\end{equation}

As usual, we denote
$(\hat{\nabla}K)(Z,X,Y):=(\hat{\nabla}_ZK)(X,Y)$, and define the
second covariant differentiation $\hat{\nabla}^2K$ of $K$ by
\begin{equation}\label{eqn:2.5}
\begin{aligned}
(\hat{\nabla}^2K)(W,Z,X,Y):=&\hat{\nabla}_W((\hat{\nabla}K)(Z,X,Y))-(\hat{\nabla}K)(\hat{\nabla}_WZ,X,Y)
\\
&-(\hat{\nabla}K)(Z,\hat{\nabla}_WX,Y)-(\hat{\nabla}K)(Z,X,\hat{\nabla}_WY).
\end{aligned}
\end{equation}
Then we have the following Ricci identity:
\begin{align}\label{eqn:2.6}
\begin{split}
&(\hat{\nabla}^2K)(W,X,Y,Z)-(\hat{\nabla}^2K)(X,W,Y,Z)\\
&=
\hat{R}(W,X)K(Y,Z)-K(\hat{R}(W,X)Y,Z)-K(Y,\hat{R}(W,X)Z).
\end{split}
\end{align}

Moreover, for unimodular affine hypersurfaces of $\mathbb{R}^{n+1}$, $K$ satisfies the so-called apolarity condition
\begin{equation}\label{eqn:2.7}
{\rm trace}\,K_X=0,\ \ \forall X\in TM.
\end{equation}

In the following, we will prove an additional relation that is very
useful in our computations. To do so, we will make use of the
technique introduced in \cite{ALVW}, as the {\it Tsinghua
Principle}. First, take the covariant derivative of \eqref{eqn:2.4} with
respect to $W$, and use \eqref{eqn:2.4} and \eqref{eqn:2.5}, to obtain
straightforwardly that
\begin{equation}\label{eqn:2.8}
(\hat{\nabla}^2K)(W,X,Y,Z)-(\hat{\nabla}^2K)(W,Y,X,Z)=0.
\end{equation}

Then we sum over cyclic permutations of the first three
vector fields in the above equation and use the Ricci identity \eqref{eqn:2.6}. It follows
that
\begin{align}\label{eqn:2.9}
\begin{split}
0=&\hat{R}(W,X)K(Y,Z)-K(\hat{R}(W,X)Z,Y)+\hat{R}(X,Y)K(W,Z)\\
&-K(\hat{R}(X,Y)Z,W)+\hat{R}(Y,W)K(X,Z)-K(\hat{R}(Y,W)Z,X).
\end{split}
\end{align}

Additionally, if $(M^n,h)=M_1^{n_1}(c_1)\times M_2^{n_2}(c_2)$ and
applying Corollary 58 on page 89 in \cite{O}, we know that
\begin{align}\label{eqn:2.10}
\begin{split}
\hat{R}(X,Y)Z=&c_1\big[h(Y_1,Z_1)X_1-h(X_1,Z_1)Y_1\big]\\
&+c_2\big[h(Y_2,Z_2)X_2-h(X_2,Z_2)Y_2\big],
\end{split}
\end{align}
where, for $p\in M^n$ and $i=1,2$, $X_i,Y_i,Z_i$ are the
$T_pM_i^{n_i}$-component of $X,Y,Z\in T_pM^n$, respectively.

\section{Lemmas on the Calculations of the
Difference Tensor}\label{sect:3}

In this section, we consider the $n$-dimensional locally strongly
convex affine hypersphere $x:M^{n}\rightarrow\mathbb{R}^{n+1}$, such
that $(M^n,h)$ is locally isometric to a Riemannian product
$M_1^{n_1}(c_1)\times M_2^{n_2}(c_2)$ for $n_1\ge2$ and $n_2\ge2$,
$n_1+n_2=n$. Here, for $i=1,2$, $M_i^{n_i}(c_i)$ denotes an
$n_i$-dimensional Riemannian manifold with constant sectional
curvature $c_i$. We first assume that $c_1^2+c_2^2\neq0$ in this
section.

Now, we would emphasize that when we dealing with the product
manifold $M_1^{n_1}\times M_2^{n_2}$, one should be aware that {\it
throughout the paper} we will work with tangent vectors on $M^n$
denoted by $X$ and $Y$. In general, the $X$ notation (as well as
$X_i$, $1\le i\le n_1$) will denote a tangent vector at
$p=(p_1,p_2)\in M^n$, with zero component on $M_2^{n_2}$. Notice
that, a priori, it means that $X$ depends on $p_2$ as well, not only
on $p_1$. A corresponding meaning is given to $Y$ (or $Y_j$, $1\le
j\le n_2$), having zero components on $M_1^{n_1}$ and depending a
priory on both $p_1$ and $p_2$. One should have in mind this meaning
when reading $X\in T_pM_1^{n_1}$, respectively, $Y\in
T_pM_2^{n_2}$.
Nonetheless, a complete understanding will be acquired with the
proofs of Lemmas \ref{lm:3.6} and \ref{lm:3.7}. 

\vskip 2mm

We begin with the following result.

\begin{lem}\label{lm:3.1}
If $c_1^2+c_2^2\neq0$, then the difference tensor $K$ vanishes
nowhere.
\end{lem}
\begin{proof}
Suppose on the contrary that the difference tensor $K$ vanishes at
the point $p=(p_1,p_2)\in M^n=M_1^{n_1}\times M_2^{n_2}$. Then, from
\eqref{eqn:2.3} we know that
\begin{equation}\label{eqn:3.1'}
\hat{R}(X,Y)Z=H\big[h(Y,Z)X-h(X,Z)Y\big]\ \ {\rm at}\ p.
\end{equation}
Thus $(M^n,h)$ has constant sectional curvature $H$ at $p$.

Now, taking unit vectors $X\in T_pM_1^{n_1}$ and $Y=Z\in
T_pM_2^{n_2}$ in both \eqref{eqn:2.10} and \eqref{eqn:3.1'}, we
get $H=0$.

Next, taking unit vectors $X,Y=Z\in T_pM_1^{n_1}$ with $X\perp Y$ in
both \eqref{eqn:2.10} and \eqref{eqn:3.1'}, we get $c_1=0$.
Similarly, taking unit vectors $X,Y=Z\in T_pM_2^{n_2}$ with
$X\perp Y$ in both \eqref{eqn:2.10} and \eqref{eqn:3.1'}, we get
$c_2=0$.

Hence, $c_1=c_2=0$. This is a contradiction to that
$c_1^2+c_2^2\neq0$.
\end{proof}

Notice that if $c_1c_2=0$, then without loss of generality we can
assume that $c_1=0$ and $c_2\neq0$. Thus, in sequel we are
sufficient to consider the following two cases:

\vskip 2mm

\begin{center}
\textbf{Case $\mathfrak{C}_1$}: $c_1=0$ and $c_2\neq 0$;\ \ \ \
\textbf{Case $\mathfrak{C}_2$}: $c_1\neq 0$ and $c_2\neq 0$.
\end{center}

\vskip 2mm

In the remaining of this section, we consider only \textbf{Case $\mathfrak{C}_1$}.
In order to decide the difference tensor, first of all we have the
following lemma.

\begin{lem}\label{lm:3.2}
For $p\in M_1^{n_1}\times M_2^{n_2}$, let $\{X_i\}_{1\le i\le n_1}$
and $\{Y_j\}_{1\le j\le n_2}$ be orthonormal bases of $T_pM_1^{n_1}$
and $T_pM_2^{n_2}$, respectively. Then, in \textbf{Case
$\mathfrak{C}_1$}, we have
\begin{equation}\label{eqn:3.2}
K_{X_i}Y_\alpha=\mu(X_i)Y_\alpha,\ \ 1\le i\le n_1,\ 1\le
\alpha\le n_2,
\end{equation}
where $\mu(X_i)=:\mu_i$ depends only on $X_i$ for $i=1,\ldots,n_1$.
Moreover, it holds that
\begin{equation}\label{eqn:3.8}
\mu(X_1)^2+\cdots+\mu(X_{n_1})^2=-\tfrac{n_1}{n_2+1}H.
\end{equation}
\end{lem}

\begin{proof}
Let $\{X_1,\ldots,X_{n_1}\}$ (resp. $\{Y_1,\ldots,Y_{n_2}\}$) be an
orthonormal basis of $T_pM_1^{n_1}$ (resp. $T_pM_2^{n_2}$). Taking
$X=X_i$, $Y=Y_\alpha$ and $Z=W=Y_\beta\ (\alpha\neq \beta)$ in
\eqref{eqn:2.9}, then using \eqref{eqn:2.10} we obtain
\begin{equation}\label{eqn:3.3}
0=c_2\sum_{m=1}^{n_2}(\delta_{\beta m}Y_\alpha-\delta_{\alpha m}Y_\beta)h(K_{X_i}Y_\beta,Y_m)
-c_2K_{X_i}Y_\alpha.
\end{equation}

Taking the component of \eqref{eqn:3.3} on $Y_\beta$, we have that
\begin{equation}\label{eqn:3.4}
h(K_{X_i}Y_\alpha,Y_\beta)=0,\ \ 1\le i\le n_1,\ 1\le \alpha\neq
\beta\le n_2.
\end{equation}

Taking the component of \eqref{eqn:3.3} on $Y_\alpha$, we have
\begin{equation}\label{eqn:3.5}
h(K_{X_i}Y_\alpha,Y_\alpha)=h(K_{X_i}Y_\beta,Y_\beta),\ \ 1\le
i\le n_1,\ 1\le \alpha,\beta\le n_2.
\end{equation}

Similarly, taking $X=Y_\alpha$, $Y=X_i$, $Z=X_j$ and $W=Y_\beta$ in
\eqref{eqn:2.9}, then using \eqref{eqn:2.10} we obtain
\begin{equation}\label{eqn:3.6}
0=c_2\sum_{m=1}^{n_2}(\delta_{m\alpha}Y_\beta-\delta_{\beta
m}Y_\alpha)h(K_{X_i}X_j,Y_m).
\end{equation}

Let $\alpha\neq \beta$, then \eqref{eqn:3.6} implies that
\begin{equation}\label{eqn:3.7}
h(K_{X_i}X_j, Y_\alpha)=0,\ \ 1\le i,j\le n_1,\ 1\le \alpha\le
n_2.
\end{equation}

Combining \eqref{eqn:3.4}, \eqref{eqn:3.5} and \eqref{eqn:3.7}, the
assertion \eqref{eqn:3.2} immediately follows.

\vskip 2mm

Next, we compute the sectional curvature $K(\pi(X_i,Y_j))$ of the
plane $\pi$ spanned by $X_i$ and $Y_j$, for some fixed
$i\in\{1,\ldots,n_1\}$ and $j\in\{1,\ldots,n_2\}$. For that purpose,
using \eqref{eqn:2.10} on the one hand, and \eqref{eqn:2.3} on the
other hand, together with applying \eqref{eqn:3.2}, we obtain
\begin{align*}
0=&H-h(K_{Y_j}Y_j, K_{X_i}X_i )+h(K_{X_i}Y_j, K_{Y_j}X_i)\\
=&H+\mu(X_i)^2-h(K_{Y_j}Y_j, K_{X_i}X_i),\ \ 1\le i\le n_1,\ 1\le
j\le n_2.
\end{align*}
Then, taking summation over $i=1,\ldots,n_1$, and using
\eqref{eqn:3.2}, we get
\begin{equation}\label{eqn:3.9}
\begin{aligned}
0&=n_1H +\sum_{i=1}^{n_1}\mu(X_i)^2-h(K_{Y_j}Y_j, \sum_{i=1}^{n_1}
K_{X_i}X_i)\\[-1mm]
&=n_1H+\sum_{i=1}^{n_1}\mu(X_i)^2-\sum_{k=1}^{n_1}\sum_{i=1}^{n_1}h(K_{X_k}
X_i, X_i)\mu(X_k).
\end{aligned}
\end{equation}

On the other hand, the apolarity condition implies that, for each
$k=1,\ldots,n_1$,
\begin{equation}\label{eqn:3.10}
0=\sum_{i=1}^{n_1}h(K_{X_k}X_i, X_i)+\sum_{j=1}^{n_2}h(K_{X_k}Y_j,
Y_j)=\sum_{i=1}^{n_1} h(K_{X_k}X_i, X_i)+n_2\mu(X_k).
\end{equation}

Therefore, from \eqref{eqn:3.9} and \eqref{eqn:3.10}, we obtain
\begin{equation}\label{eqn:3.11}
\mu(X_1)^2+\cdots+\mu(X_{n_1})^2=-\tfrac{n_1}{n_2+1}H.
\end{equation}

This completes the proof of Lemma \ref{lm:3.2}.
\end{proof}

Now, before going to show the next lemma, we will describe the
construction of a typical orthonormal basis, which was introduced by
N. Ejiri and has been widely applied, and proved to be very useful
for various situations, see e.g. \cite{HLSV} and \cite{LV,MU}. The
idea is to construct a basis from a self-adjoint operator at a
point; then one extends the basis to local orthonormal vector
fields. In this paper, we have the general principle as below:

For an arbitrary $p\in M^n=M_1^{n_1}\times M_2^{n_2}$, let
$U_pM_1^{n_1}=\{u\in T_pM_1^{n_1} \mid h(u,u)=1\}$ and $E_p\subset
T_{p_1}M_1^{n_1}\times \{0\}$ a vector subspace. Since $M^n$ is
locally strongly convex, $U_pM_1^{n_1}\cap E_p$ is compact. We
define on this set the function
$$
f_1(u)=h(K_uu,u),\ \ u\in U_pM_1^{n_1}\cap E_p.
$$
Then there is an element $e_{1}\in U_pM_1^{n_1}\cap E_p$ at which the function
$f_1(u)$ attains the absolute maximum. Let $u\in U_pM_1^{n_1}\cap E_p$ such that
$h(u,e_1)=0$, and define a function $g$ by $g(t):=f_1\big(\cos t\,e_1+\sin t\,
u\big)$. Then we have
\begin{equation}\label{eqn:3.12}
g^\prime(0)=3\,h(K_{e_1}e_1,u),\ \
g^{\prime\prime}(0)=6\,h(K_{e_1}u,u)-3\,f_1(e_1).
\end{equation}

Since $g$ attains an absolute maximum at $t=0$, we have
$g^\prime(0)=0, g^{\prime\prime}(0)\le 0$, i.e.,
\begin{equation}\label{eqn:3.13}
h(K_{e_1}e_1,u)=0,\ h(K_{e_1}e_1,e_1)\ge 2h(K_{e_1}u,u),\
h(u,u)=1,\, u\perp e_1.
\end{equation}

Analogously, we can define a function $f_2$ on $U_pM_2^{n_2}\cap
\tilde E_p$, where $U_pM_2^{n_2}=\{u\in T_pM_2^{n_2} \mid
h(u,u)=1\}$ and $\tilde E_p\subset \{0\}\times T_{p_2}M_2^{n_2}$ a
vector subspace. We can choose $e_1\in U_pM_2^{n_2}\cap \tilde E_p$
such that \eqref{eqn:3.13} holds for $u\in U_pM_2^{n_2}\cap \tilde
E_p$ with $u\perp e_1$.

In the following, we will apply the above principle of choosing the
unit vector $e_1$ many times.

\vskip 1mm

Now, as a supplement to Lemma \ref{lm:3.2}, we can prove the
following lemma.

\begin{lem}\label{lm:3.4}
Given $p=(p_1,p_2)\in M_1^{n_1}\times M_2^{n_2}$. Let $\{X_i\}_{1\le
i\le n_1}$ and $\{Y_j\}_{1\le j\le n_2}$ be the orthonormal bases of
$T_pM_1^{n_1}$ and $T_pM_2^{n_2}$, respectively. Then, in
\textbf{Case $\mathfrak{C}_1$}, we have
\begin{equation}\label{eqn:3.14}
K_{Y_\alpha}Y_\beta=\delta_{\alpha\beta}(\mu_1X_1+\cdots+\mu_{n_1}X_{n_1}),\
\ 1\le \alpha,\beta\le n_2,
\end{equation}
Moreover, we have $c_2=\frac{n+1}{n_2+1}H<0$.
\end{lem}

\begin{proof} Let $\{X_i\}_{1\le
i\le n_1}$ and $\{Y_j\}_{1\le j\le n_2}$ be orthonormal bases of
$T_pM_1^{n_1}$ and $T_pM_2^{n_2}$, respectively. Then, according to
Lemma \ref{lm:3.2}, there are constants
$\{\theta_{\alpha\beta}^\gamma\}$ such that
\begin{equation*}
K_{Y_\alpha}Y_\beta=\delta_{\alpha\beta}(\mu_1
X_1+\cdots+\mu_{n_1}X_{n_1})+\sum_{\gamma=1}^{n_2}\theta_{\alpha\beta}^\gamma Y_l,\ \
1\le\alpha,\beta\le n_2.
\end{equation*}

We will show that $\theta_{\alpha\beta}^\gamma=0$ for $1\le \alpha,\beta,\gamma\le n_2$, or equivalently,
\begin{equation}\label{eqn:3.15}
h(K_{Y_\alpha}Y_\beta, Y_\gamma)=0,\ \ 1\le \alpha,\beta,\gamma\le n_2.
\end{equation}

We will prove \eqref{eqn:3.15} by contradiction.

Suppose on the contrary that \eqref{eqn:3.15} does not hold. Then,
following the preceding stated procedure, we can choose a unit
vector in $U_pM_2^{n_2}$, denoted by $\bar Y_1$, such that
$\theta_1:=h(K_{\bar Y_1}\bar Y_1,\bar Y_1)>0$ is the maximum of the
function $f_2$ defined on $U_{p_2}M_2^{n_2}$.

Define an operator $\mathcal{A}:T_pM_2^{n_2}\to
T_pM_2^{n_2}$ by
$$
\mathcal{A}(Y):=K_{\bar Y_1}Y-h(K_{\bar
Y_1}Y,X_1)X_1-\cdots-h(K_{\bar Y_1}Y,X_{n_1})X_{n_1}.
$$

Then, it is easy to show that $\mathcal{A}$ is self-adjoint and
satisfies $\mathcal{A}(\bar Y_1)=\theta_1\bar Y_1$. We can choose
orthonormal vectors in $U_pM_2^{n_2}$ orthogonal to $\bar Y_1$,
denoted by $\bar Y_2$, $\ldots, \bar Y_{n_2}$, which are the
remaining eigenvectors of the operator $\mathcal{A}$, with
associated eigenvalues $\theta_2,\ldots,\theta_{n_2}$, respectively.
Thus, by Lemma \ref{lm:3.2}, we get the conclusion that
\begin{equation}\label{eqn:3.16}
K_{\bar Y_1}\bar Y_1=\mu_1X_1+\cdots+\mu_{n_1}X_{n_1}+\theta_1\bar
Y_1,\ \ K_{\bar Y_1}\bar Y_i=\theta_i\bar Y_i,\ \ 2\le i\le n_2.
\end{equation}

In order to solve $\{\theta_i\}$ in \eqref{eqn:3.16}, taking
$X=Z=\bar Y_1$ and $Y=\bar Y_i$, $2\le i\le n_2$, in
\eqref{eqn:2.3}, using \eqref{eqn:2.10}, \eqref{eqn:3.16} and Lemma
\ref{lm:3.2}, we can obtain
\begin{equation}\label{eqn:3.17}
\theta_i^2-\theta_1\theta_i+\tfrac{n_1+1}{n_2+1}H-c_2=0,\ \ 2\le
i\le n.
\end{equation}

From \eqref{eqn:3.17} and the statement of \eqref{eqn:3.13}, we obtain
that
\begin{equation}\label{eqn:3.18}
\theta_2=\cdots=\theta_{n_2}=\tfrac12\big(\theta_1-\sqrt{\theta_1^2
-4(\tfrac{n+1}{n_2+1}H-c_2)}\,\big).
\end{equation}

Using \eqref{eqn:3.2}, \eqref{eqn:3.16}, \eqref{eqn:3.18} and
${\rm trace}\,K_{Y_1}=0$, we get
\begin{equation}\label{eqn:3.19}
(n_2+1)\theta_1=(n_2-1)\sqrt{\theta_1^2
-4(\tfrac{n+1}{n_2+1}H-c_2)}.
\end{equation}

Then, we have
$$
4\left(c_2-\tfrac{n+1}{n_2+1}H\right)=\Big[\left(\tfrac{n_2+1}{n_2-1}\right)^2-1\Big]\theta_1^2>0.
$$

It follows that $c_2>\tfrac{n+1}{n_2+1}H$ and
\begin{equation}\label{eqn:3.20}
\theta_1=(n_2-1)\sqrt{\tfrac{(n_2+1)c_2-(n+1)H}{n_2(n_2+1)}}.
\end{equation}

Next, we intend to extend $\bar Y_1\in U_pM_2^{n_2}$, that
satisfying \eqref{eqn:3.16}, to be a local unit vector field around
$p\in M^n$. For that purpose, we first make the following Claim.

\vskip 2mm

\textbf{Claim 1}. {\it For every $p=(p_1,p_2)\in M^n$, the set
$$
\Omega_p:=\Big\{\lambda\in\mathbb R\mid
V\in U_pM_2^{n_2}\ {\rm s.\,t.\ } K_{V}V=\lambda
V+\sum_{i=1}^{n_1}\mu_iX_i\Big\}
$$
consists of finite numbers, which are independent of the point $p\in M^n$.}

\vskip 2mm

To verify the claim, we notice that, for any fixed $p\in M^n$, the
above discussion implies that we have $\theta_1\in\Omega_p$ with
$V=\bar Y_1$. Thus, the set $\Omega_p$ is non-empty.

Next, assume an arbitrary $\lambda\in\Omega_p$ associated with $V\in
U_pM_2^{n_2}$ such that
$$
K_VV=\lambda V+\mu_1X_1+\cdots+\mu_{n_1}X_{n_1}.
$$

Then we put $\tilde Y_1=V$, $\tilde\theta_1=\lambda$ and define an
operator $\mathcal{B}:T_pM_2^{n_2}\to T_pM_2^{n_2}$ by
$$
\mathcal{B}(Y)=K_{\tilde Y_1}Y-h(K_{\tilde
Y_1}Y,X_1)X_1-\cdots-h(K_{\tilde Y_1}Y,X_{n_1})X_{n_1}.
$$

It is easily seen that $\mathcal{B}$ is self-adjoint and
$\mathcal{B}(\tilde Y_1)=\tilde\theta_1\tilde Y_1$. Then, we may
complete $\tilde Y_1$ to get an orthonormal basis $\{\tilde
Y_i\}_{1\le i\le n_2}$ of $T_pM_2^{n_2}$ by letting $\tilde
Y_2,\ldots, \tilde Y_{n_2}$ to be the eigenvectors of $\mathcal{B}$,
with eigenvalues $\tilde\theta_2,\ldots, \tilde\theta_{n_2}$,
respectively.

Similar to the proof of \eqref{eqn:3.17}, we have the existence of
an integer $n_{2,1}$ with $0\le n_{2,1}\le n_2-1$ such that, if
necessary, after renumbering the basis, it holds
\begin{equation}\label{eqn:3.21}
\left\{
\begin{aligned}
&\tilde\theta_2=\cdots=\tilde
\theta_{n_{2,1}+1}=\tfrac12\left(\tilde\theta_1+\sqrt{\tilde
\theta_1^2
-4(\tfrac{n+1}{n_2+1}H-c_2)}\,\right),\\
&\tilde\theta_{n_{2,1}+2}=\cdots=\tilde
\theta_{n_2}=\tfrac12\left(\tilde\theta_1-\sqrt{\tilde\theta_1^2
-4(\tfrac{n+1}{n_2+1}H-c_2)}\,\right).
\end{aligned}
\right.
\end{equation}

Then, by ${\rm trace}\,K_{\tilde Y_1}=0$, we find that
\begin{equation}\label{eqn:3.22}
(n_2+1)\tilde\theta_1-(n_2-2n_{2,1}-1)\sqrt{\tilde\theta_1^2
-4(\tfrac{n+1}{n_2+1}H-c_2)}=0.
\end{equation}

This implies that $\tilde\theta_1=\lambda$ is independent of the
point $p$ and takes value of only finite possibilities. The
assertion of Claim 1 immediately follows.

\vskip 2mm

To extend $\bar Y_1$ differentiably to a unit vector field on a
neighbourhood $U\subset M^n$ around $p$, which is still denoted by
$\bar Y_1$, such that, at every point $q\in U$, $f_2$ attains an
absolute maximum at $\bar Y_1(q)$,
we first take differentiable $h$-orthonormal vector fields
$\{E_1,\ldots,E_{n_2}\}$ defined on a neighbourhood $U'$ of $p$ and
satisfying $E_i(q)\in T_qM_2^{n_2}, q\in U',1\le i\le n_2$, such
that $E_i(p)=\bar Y_i$ for $1\le i\le n_2$. Then, we define a
function $\gamma$ by
$$
\gamma:\mathbb{R}^{n_2}\times U'\rightarrow \mathbb{R}^{n_2}\ {\rm
by}\ (a_1,\ldots,a_{n_2},q)\mapsto(b_1,\ldots,b_{n_2}),
$$
where
\begin{equation}\label{eqn:3.22add1}
b_k=\sum_{i,j=1}^{n_2}a_ia_jh(K_{E_i}E_j,E_k)-\theta_1a_k,\ \ 1\le
k\le n_2,
\end{equation}
are regarded as functions on $\mathbb{R}^{n_2}\times U'$:
$b_k=b_k(a_1,\ldots,a_{n_2},q)$.

Using \eqref{eqn:3.16} and the fact that $f_2$ attains an
absolute maximum at $E_1(p)$, we then obtain that
\begin{equation*}
\begin{aligned}
\tfrac{\partial b_k}{\partial a_m}(1,0,\ldots,0,p)&=2h(K_{E_1(p)}E_m(p),E_k(p))-\theta_1\delta_{km}\\
&=\left\{
\begin{aligned}
&0,\ \ {\rm if}\  k\neq m, \\
&\theta_1,\  \ {\rm if}\ k=m=1,\\
&2\theta_k-\theta_1,\ \ {\rm if}\ k= m\ge2.
\end{aligned}
\right.
\end{aligned}
\end{equation*}

Notice that, by assumption, \eqref{eqn:3.18} and \eqref{eqn:3.19},
we have $\theta_1>0$ and $2\theta_k-\theta_1\neq0$ for $2\le k\le
n_2$. Then, the implicit function theorem shows that there exist
differentiable functions $\{a_i(q)\}_{1\le i\le n_2}$ defined on a
neighbourhood $U''\subset U'$ of $p$, such that
\begin{equation}\label{eqn:3.22add2}
\left\{
\begin{aligned}
&a_1(p)=1,\ a_2(p)=\cdots=a_{n_2}(p)=0,\\
&b_i(a_1(q),\ldots,a_{n_2}(q),q)\equiv0,\ \ 1\le i\le n_2.
\end{aligned}
\right.
\end{equation}

Define the local vector field $V$ on $U''$ by
$$
V(q)=a_1(q)E_1(q)+\cdots+a_{n_2}(q)E_{n_2}(q),\ \
q\in U''.
$$
Then, for local basis of $TM_1^{n_1}$ around $U''$, still denoted by
$\{X_i\}_{1\le i\le n_1}$, from \eqref{eqn:3.22add1},
\eqref{eqn:3.22add2} and Lemma \ref{lm:3.2}, we have
$K_{X_i}Y=\mu_iY$ for any $Y\in TM_2^{n_2}$, and that
$$
K_VV=\theta_1 V+\mu_1h(V,V)X_1+\cdots+\mu_{n_1}h(V,V)X_{n_1}.
$$

Let us define $\|V\|=\sqrt{h(V,V)}$. Since $\|V\|(p)=1$, there
exists a neighbourhood $U\subset U''$ of $p$ such that $V\not=0$
on $U$. Then, $W=\tfrac{V}{\|V\|}$ is a unit vector field on $U$
that satisfies
$$
K_WW
=\tfrac{\theta_1}{\sqrt{h(V,V)}}W+\mu_1X_1+\cdots+\mu_{n_1}X_{n_1}.
$$

Denote $\tilde\theta_1=\theta_1/\sqrt{h(V,V)}$. Then, the proof of
Claim 1 implies that, as a function on $U$, $\tilde\theta_1$ takes
values of finite number, which satisfy \eqref{eqn:3.22} for some
$0\le n_{2,1}\le n_2-1$. This further implies from the fact
$h(V,V)(p)=1$ and the continuity of the function
$\theta_1/\sqrt{h(V,V)}$ that $h(V,V)\equiv1$ on $U$.

Let $\bar Y_1=W$ and take orthonormal vector fields $\bar
Y_2,\ldots,\bar Y_{n_2}$ orthogonal to $\bar Y_1$ so that $\{\bar
Y_1,\ldots,\bar Y_{n_1}\}$ forms a local orthonormal basis of
$TM_2^{n_2}$ on $U$. Then, according to \eqref{eqn:3.16},
\eqref{eqn:3.18} and \eqref{eqn:3.20}, we have a constant
$\theta_2=\cdots=\theta_{n_2}$ such that the difference tensor
satisfies
\begin{equation}\label{eqn:3.23}
K_{\bar Y_1}\bar Y_1=\mu_1X_1+\cdots+\mu_{n_1}X_{n_1}+\theta_1\bar
Y_1,\ \ K_{\bar Y_1}\bar Y_i=\theta_i\bar Y_i,\ \ 2\le i\le n_2.
\end{equation}

\vskip 2mm

Now, we can apply the Codazzi equation \eqref{eqn:2.4} to the basis
$\{\bar Y_i\}_{1\le i\le n_2}$.

By the property $h(\hat{\nabla}_{\bar Y_i}\bar Y_j,X_k)=0$ of
product manifold and \eqref{eqn:3.23}, we have the following
calculations:
\begin{equation}\label{eqn:3.24}
\begin{aligned}
(\hat{\nabla}_{\bar Y_i}K)(\bar Y_1,\bar Y_1)&=\hat{\nabla}_{\bar Y_i}K(\bar Y_1,\bar Y_1)-2K(\hat{\nabla}_{\bar Y_i}\bar Y_1,\bar Y_1)\\[-2mm]
&=(\theta_1-2\theta_2)\hat{\nabla}_{\bar Y_i}\bar Y_1+\sum_{k=1}^{n_1}\Big(\mu_k\hat{\nabla}_{\bar Y_i}X_k+\bar Y_i(\mu_k)X_k\Big),\\[-2mm]
&=(\theta_1-2\theta_2)\sum_{j=1}^{n_2}h(\hat{\nabla}_{\bar Y_i}\bar
Y_1,\bar Y_j)\bar Y_j\\[-2mm]
&\hspace{3cm}+\sum_{k=1}^{n_1}\Big(\mu_k \hat{\nabla}_{\bar
Y_i}X_k+\bar Y_i(\mu_k)X_k\Big),
\end{aligned}
\end{equation}
\begin{equation}\label{eqn:3.25}
\begin{aligned}
(\hat{\nabla}_{\bar Y_1}K)(\bar Y_i,\bar Y_1)&=\hat{\nabla}_{\bar
Y_1}K(\bar Y_i,\bar Y_1)
-K(\hat{\nabla}_{\bar Y_1}\bar Y_i,\bar Y_1)-K(\hat{\nabla}_{\bar Y_1}\bar Y_1,\bar Y_i)\\[-1mm]
&=\theta_2\hat{\nabla}_{\bar Y_1}\bar Y_i-K(\hat{\nabla}_{\bar Y_1}\bar Y_i,\bar Y_1)-K(\hat{\nabla}_{\bar Y_1}\bar Y_1,\bar Y_i)\\[-1mm]
&=\theta_2h(\hat{\nabla}_{\bar Y_1}\bar Y_i,\bar Y_1)\bar Y_1-h(\hat{\nabla}_{\bar Y_1}\bar Y_i,\bar Y_1)K(\bar Y_1,\bar Y_1)\\[-1mm]
&\hspace{32mm} -\sum_{j=2}^{n_2}h(\hat{\nabla}_{\bar Y_1}\bar
Y_1,\bar Y_j)K(\bar Y_j,\bar Y_i).
\end{aligned}
\end{equation}

Then, using $h((\hat{\nabla}_{\bar Y_i}K)(\bar Y_1,\bar Y_1),\bar
Y_1)=h((\hat{\nabla}_{\bar Y_1}K)(\bar Y_i,\bar Y_1),\bar Y_1)$ for
$i\ge2$ we get $\hat{\nabla}_{\bar Y_1}\bar Y_1=0$. This and
\eqref{eqn:3.25} give that $(\hat{\nabla}_{\bar Y_1}K)(\bar Y_i,\bar
Y_1)=0$ for $1\le i\le n_2$. Thus, using \eqref{eqn:2.4} and
\eqref{eqn:3.24}, we can finally get
\begin{equation}\label{3.26}
\hat{\nabla}_{\bar Y_i}\bar Y_1=0,\ \ 1\le i\le n_2.
\end{equation}

It follows that $c_2=h(\hat{R}(\bar Y_2,\bar Y_1)\bar Y_1,\bar
Y_2)=0$ and as desired we get a contradiction. Therefore,
\eqref{eqn:3.15} does hold.

Finally, taking $X=\bar Y_2$ and $Y=Z=\bar Y_1$ in \eqref{eqn:2.3},
with using \eqref{eqn:2.10}, \eqref{eqn:3.2} and \eqref{eqn:3.14},
we easily get the relation $c_2=\frac{n+1}{n_2+1}H$. This together
with \eqref{eqn:3.8} further implies that $H<0$.

We have completed the proof of Lemma \ref{lm:3.4}.
\end{proof}

For the difference tensor, besides the conclusions as stated in
Lemmas \ref{lm:3.2} and \ref{lm:3.4}, we shall construct in the
following Lemma \ref{lm:3.5} a typical local orthonormal frame on
$M^n$ so that more information of the difference tensor can be
derived for \textbf{Case $\mathfrak{C}_1$}. However, the proof of
Lemma \ref{lm:3.5} becomes more complicated when we compare it with
that of Lemma \ref{lm:3.4}.

\begin{lem}\label{lm:3.5}
In \textbf{Case $\mathfrak{C}_1$}, given $p\in M^n$, there exist
local orthonormal vector fields $\{X_i\}_{1\le i\le n_1}$ defined on
a neighbourhood $U$ of $p$, and satisfying $X_i(q)\in T_qM_1^{n_1}$
for $q\in U$ and $1\le i\le n_1$, such that the difference tensor
$K$ takes the following form:
\begin{equation}\label{eqn:3.27}
\left\{
\begin{aligned}
&K_{X_1}X_1=\lambda_{1,1}X_1,\\
&K_{X_i}X_i=\mu_1 X_1+\cdots+\mu_{i-1} X_{i-1}+\lambda_{i,i}X_i,\ \ 2\le i\le n_1,\\
&K_{X_i}X_j=\mu_iX_j,\ \ 1\le i<j\le n_1,\\
& K_{X_i}Y=\mu_iY,\ Y(q)\in T_qM_2^{n_2},\ \ 1\le i\le n_1,
\end{aligned}
\right.
\end{equation}
where $\lambda_{i,i}$ and $\mu_i\ (1\le i\le n_1)$ are constants,
and they satisfy the relations
\begin{equation}\label{eqn:3.28}
\left\{
\begin{aligned}
&\lambda_{i,i}+(n-i)\mu_i=0,\ \ 1\le i\le n_1,\\
&\lambda_{i,i}>0,\ 1\le i\le n_1-1;\ \ \lambda_{n_1,n_1}\ge0.
\end{aligned}
\right.
\end{equation}
\end{lem}

\begin{proof}

We give the proof by induction on the subscript $i$ of $K_{X_i}$.
According to the general principle of induction method, this
consists of two steps as below.

\vskip 2mm

\noindent {\bf The first step of induction.}

In this step, we should verify the assertion for $i=1$. To do so, we
have to show that, around any given $p\in M_1^{n_1}\times
M_2^{n_2}$, there exist orthonormal vector fields $\{X_i\}_{1\le
i\le n_1}$ defined on a neighbourhood $U$ of $p$ and satisfying
$X_i(q)\in T_qM_1^{n_1}$ for $q\in U$ and $1\le i\le n_1$, and real
numbers $\lambda_{1,1}>0$ and $\mu_1$, so that we have
$$
\left\{
\begin{aligned}
&K_{X_1}X_1=\lambda_{1,1} X_1,\ \
K_{X_1}X_i=\mu_1 X_i,\ \ 2\le i\le n_1,\\
&K_{X_1}Y=\mu_1 Y,\ \ Y(q)\in T_qM_2^{n_2},\\
&\lambda_{1,1}+(n-1)\mu_1=0.
\end{aligned}
\right.
$$
The proof of the above assertion will be divided into four claims as
below.

\vskip 2mm

\noindent\textbf{Claim I-(1)}. {\it Given $p\in
M_1^{n_1}\times M_2^{n_2}$, there exists an orthonormal basis
$\{X_i\}_{1\le i\le n_1}$ of $T_pM_1^{n_1}$, real numbers
$\lambda_{1,1}>0$, $\lambda_{1,2}=\cdots=\lambda_{1,n_1}$ and
$\mu_1$, such that $\lambda_{1,1}$ is the maximum of $f_1$
defined on $U_pM_1^{n_1}$, and the following relations hold:}
\begin{equation}\label{eqn:3.28add}
\left\{
\begin{aligned}
&K_{X_1}X_1=\lambda_{1,1} X_1,\ K_{X_1}X_i=\lambda_{1,i} X_i,\ \
2\le i\le n_1,\\[1mm]
&K_{X_1}Y=\mu_1 Y,\ \ Y\in T_pM_2^{n_2}.
\end{aligned}
\right.
\end{equation}
\begin{proof}[Proof of \textbf{Claim I-(1)}]
First, if for an orthonormal vectors $\{X_i\}_{1\le i\le n_1}$ and
for any $i,j,k=1,\ldots, n_1$, it holds $h(K_{X_i}X_j,X_k)=0$. Then
in \eqref{eqn:2.3} taking $X=X_1$ and $Y=Z=X_2$, using
\eqref{eqn:2.10} and \eqref{eqn:3.2}, we obtain $H=0$. This is a
contradiction to Lemma \ref{lm:3.4}.

Next, let $p\in M^n=M_1^{n_1}(c_1)\times M_2^{n_2}(c_2)$. We
choose $X_1\in U_pM_1^{n_1}$ such that
$\lambda_{1,1}=h(K_{X_1}X_1,X_1)$ is the maximum of $f_1(u)$ on
$U_pM_1^{n_1}$ and it must be the case $\lambda_{1,1}>0$. Then,
according to \eqref{eqn:3.2} and the statement of \eqref{eqn:3.13},
we know that $X_1$ is an eigenvector of $K_{X_1}$ and we can choose
orthonormal vectors $X_2,\ldots,X_{n_1}\in T_pM_1^{n_1}$
orthogonal to $X_1$ such that $K_{X_1}X_i=\lambda_{1,i}X_i$ for
$1\le i\le n_1$, and $K_{X_1}Y=\mu_1 Y$ for any $Y\in T_pM_2^{n_2}$.


Taking in \eqref{eqn:2.3} $X=Z=X_1$ and $Y=X_k$, and using
\eqref{eqn:2.10}, we can obtain
\begin{equation}\label{eqn:3.29}
\lambda_{1,k}^2-\lambda_{1,1}\lambda_{1,k}+H=0,\ \ 2\le k\le n_1.
\end{equation}

Similar to the proof of \eqref{eqn:3.13}, we have $\lambda_{1,1}\ge
2\lambda_{1,k}$ for $2\le k\le n_1$. Thus, solving \eqref{eqn:3.29} we
obtain $\lambda_{1,2}=\cdots=\lambda_{1,n_1}$ with
\begin{equation}\label{eqn:3.30}
\lambda_{1,k}=\tfrac{1}{2}\Big(\lambda_{1,1}-\sqrt{\lambda_{1,1}^2-4H}\,\Big),\
\ 2\le k\le n_1.
\end{equation}

Furthermore, taking in \eqref{eqn:2.3} $X=Z=X_1$ and $Y\in T_pM_2^{n_2}$ be a unit vector,
using \eqref{eqn:2.10} and \eqref{eqn:3.2}, we get
\begin{equation}\label{eqn:3.31}
\mu_1^2-\mu_1\lambda_{1,1}+H=0.
\end{equation}
Hence we have
\begin{equation}\label{eqn:3.32}
\mu_1
=\tfrac{1}{2}\Big(\lambda_{1,1}+\varepsilon_1\sqrt{\lambda_{1,1}^2-4H}\,\Big),\
\ \varepsilon_1=\pm 1.
\end{equation}

Finally, by \eqref{eqn:3.2}, \eqref{eqn:3.30}, \eqref{eqn:3.32} and
${\rm trace}\,K_{X_1}=0$, we obtain
\begin{equation}\label{eqn:3.33}
(n+1)\lambda_{1,1}+(-n_1+1+\varepsilon_1n_2)\sqrt{\lambda_{1,1}^2-4H}=0,
\end{equation}
and therefore, we have
\begin{equation}\label{eqn:3.34}
\lambda_{1,1}=2\sqrt{\tfrac{-H}{(\tfrac{n+1}{n_1-\varepsilon_1
n_2-1})^2-1}}.
\end{equation}

From \eqref{eqn:3.30}, \eqref{eqn:3.32} and \eqref{eqn:3.34}, we
have completed the proof of \noindent\textbf{Claim I-(1)}.
\end{proof}

\vskip 2mm

\noindent\textbf{Claim I-(2)}. {\it The real numbers described in
Claim I-(1) satisfy the relations:}
$$
\lambda_{1,2}=\cdots=\lambda_{1,n_1}=\mu_1\ \ {\rm and}\ \
\lambda_{1,1}+(n-1)\mu_1=0.
$$

\begin{proof}[Proof of \textbf{Claim I-(2)}]
From \eqref{eqn:3.30}, \eqref{eqn:3.32} and ${\rm
trace}\,K_{X_1}=0$, the assertions are equivalent to that
$\varepsilon_1=-1$. Suppose on the contrary that $\varepsilon_1=1$.
Then we have
\begin{equation}\label{eqn:3.37}
\mu_1\lambda_{1,2}=H,
\end{equation}
and \eqref{eqn:3.33} implies that
\begin{equation}\label{eqn:3.36}
n_1>n_2+1\ge3.
\end{equation}

Put $V_1=\{u\in T_pM_1^{n_1}|\,u\perp X_1\}$. Then, by arguments as
in the beginning of the proof for Claim I-(1) shows that the
function $f_1\not=0$ restricting on $V_1\cap U_pM_1^{n_1}$. We
rechoose a unit vector $X_2\in V_1$ such that
$\lambda_{2,2}=h(K_{X_2}X_2,X_2)>0$ is the maximum of $f_1(u)$
restricted on $\{u\in U_pM_1^{n_1}|\,u\perp X_1\}$.

Then, according to Lemma \ref{lm:3.2}, we can define a linear
mapping $\mathcal{A}:\ V_1\to V_1$ by
$\mathcal{A}(X):=K_{X_2}X-h(K_{X_2}X,X_1)X_1$. It is easily seen
that $\mathcal{A}$ is self-adjoint and $X_2$ is one of its
eigenvector. We can choose orthonormal vectors
$X_3,\ldots,X_{n_1}\in T_pM_1^{n_1}$ orthogonal to $X_2$, which are
the remaining eigenvectors of the operator $\mathcal{A}$, associated
to the eigenvalues $\lambda_{2,3},\ldots, \lambda_{2,n_1}$,
respectively. Therefore, we have
\begin{equation}\label{eqn:3.38}
K_{X_2}X_2=\lambda_{1,2}X_1+\lambda_{2,2}X_2,\
K_{X_2}X_i=\lambda_{2,i}X_i,\ \ 3\le i\le n_1.
\end{equation}

\vskip 1mm

Now, we can make use of \eqref{eqn:3.38} to derive the expected
contradiction.

Taking in \eqref{eqn:2.3} $X=Z=X_2$ and $Y=X_k$, using
\eqref{eqn:2.10} and \eqref{eqn:3.38}, we can obtain
\begin{equation}\label{eqn:3.39}
\lambda_{2,k}^2-\lambda_{2,2}\lambda_{2,k}+H-\lambda_{1,2}^2=0,\ \
3\le k\le n_1.
\end{equation}

Similar to the proof of \eqref{eqn:3.13}, we have $\lambda_{2,2}\ge
2\lambda_{2,k}$ for $3\le k\le n_1$. Then, solving \eqref{eqn:3.13},
we get $\lambda_{2,3}=\cdots=\lambda_{2,n_1}$ with
\begin{equation}\label{eqn:3.40}
\lambda_{2,k}=\tfrac{1}{2}\Big(\lambda_{2,2}-\sqrt{\lambda_{2,2}^2-4(H-\lambda_{1,2}^2)}\,\Big),\
\ 3\le k\le n_1.
\end{equation}

Similarly, taking in \eqref{eqn:2.3} $X=Z=X_2$ and $Y\in T_pM_2^{n_2}$
a unit vector, using \eqref{eqn:2.10}, \eqref{eqn:3.2} and
\eqref{eqn:3.38}, we get
\begin{equation}\label{eqn:3.41}
\mu_2^2-\mu_2\lambda_{2,2}+H-\lambda_{1,2}\mu_1=0.
\end{equation}

By using \eqref{eqn:3.37}, we can reduce \eqref{eqn:3.41} to be
\begin{equation}\label{eqn:3.42}
\mu_2^2-\mu_2\lambda_{2,2}=0.
\end{equation}
It follows that
\begin{equation}\label{eqn:3.43}
\mu_2=\tfrac{1}{2}(\lambda_{2,2}+\varepsilon_2\lambda_{2,2}),\ \
\varepsilon_2=\pm1.
\end{equation}

Then, by ${\rm trace}\,K_{X_2}=0$, and using \eqref{eqn:3.2},
\eqref{eqn:3.38}, \eqref{eqn:3.40} and \eqref{eqn:3.43}, we have
\begin{equation}\label{eqn:3.44}
(n_1+n_2+\varepsilon_2n_2)\lambda_{2,2}=(n_1-2)\sqrt{\lambda_{2,2}^2+4(\lambda_{1,2}^2-H)},
\end{equation}
which implies that
\begin{equation}\label{eqn:3.45}
\lambda_{2,2}=\sqrt{\tfrac{4(\lambda_{1,2}^2-H)}{\big(\tfrac{n_1+n_2+\varepsilon_2n_2}{n_1-2}\big)^2-1}}.
\end{equation}

Note that $\varepsilon_1=1$, from \eqref{eqn:3.34} we have
\begin{equation}\label{eqn:3.46}
\lambda_{1,1}=\sqrt{\tfrac{-4H}{\big(\tfrac{n_1+n_2+1}{n_1-n_2-1}\big)^2-1}}.
\end{equation}

Noticing that $n_2\ge2$ and, by \eqref{eqn:3.36}, $n_1\ge n_2+2$,
we have
\begin{equation*}
\begin{aligned}
&\tfrac{n_1+n_2+1}{n_1-n_2-1}-\tfrac{n_1+n_2+\varepsilon_2n_2}{n_1-2}
>\tfrac{n_1+n_2+1}{n_1-n_2-1}-\tfrac{n_1+2n_2}{n_1-2}=\tfrac{2(n_2+1)(n_2-1)}{(n_1-n_2-1)(n_1-2)}>0.
\end{aligned}
\end{equation*}
This, together with $H<0$, implies that
$\lambda_{2,2}>\lambda_{1,1}$. This is a contradiction.

Hence, we have $\varepsilon_1=-1$ and
$\lambda_{1,2}=\cdots=\lambda_{1,n_1}=\mu_1$.

Then, by ${\rm trace}\,K_{X_1}=0$ we get the second assertion.
\end{proof}

\noindent\textbf{Claim I-(3)}. {\it For every point $p=(p_1,p_2)\in
M^n$, the set
$$
\Omega_p:=\Big\{\lambda\in\mathbb R\mid V\in U_pM_1^{n_1}\ {\rm
s.\,t.\ } K_{V}V=\lambda V\Big\}
$$
consists of finite numbers, which are independent of $p\in M^n$.}

\begin{proof}[Proof of \textbf{Claim I-(3)}]
Claim I-(1) implies that $\Omega_p$ is non-empty. Assume that there
exists a unit vector $V\in T_pM_1^{n_1}$ such that $K_{V}V=\lambda
V$. Let $X_1:=V$ and $\lambda_{1,1}=\lambda$. Then, according to
Lemma \ref{lm:3.2}, we may complete $X_1$ to obtain an orthonormal
basis $\{X_i\}_{1\le i\le n_1}$ of $T_pM_1^{n_1}$ such that, for
each $2\le k\le n_1$, $X_k$ is the eigenvector of $ K_{X_1}$ with
eigenvalue $\lambda_{1,k}$.

Then we have \eqref{eqn:3.29}, from which we have an integer
$n_{1,1}$, $0\le n_{1,1}\le n_1-1$, such that, if necessary after
renumbering the basis, we have
\begin{equation}\label{eqn:3.47}
\left\{
\begin{aligned}
&\lambda_{1,2}=\cdots=\lambda_{1,n_{1,1}+1}=\tfrac{1}{2}\big(\lambda_{1,1}+\sqrt{\lambda_{1,1}^2-4H}\,\big),\\
&\lambda_{1,n_{1,1}+2}=\lambda_{1,n_1}=\tfrac{1}{2}\big(\lambda_{1,1}-\sqrt{\lambda_{1,1}^2-4H}\,\big).
\end{aligned}
\right.
\end{equation}

Similarly, we have \eqref{eqn:3.32}. Then, by ${\rm
trace}\,K_{X_1}=0$, we have
\begin{equation}\label{eqn:3.48}
(n+1)\lambda_{1,1}+(2n_{1,1}-n_1+1+\varepsilon_1n_2)\sqrt{\lambda_{1,1}^2-4H}=0.
\end{equation}

If $2n_{1,1}-n_1+1+\varepsilon_1n_2=0$, then $\lambda_{1,1}=0$.

If $2n_{1,1}-n_1+1+\varepsilon_1n_2<0$, then we have
\begin{equation}\label{eqn:3.49}
\lambda_{1,1}=\sqrt{\tfrac{4H}{1-(\tfrac{n_1+n_2+1}{2n_{1,1}-n_1+\varepsilon_1 n_2+1})^2}}.
\end{equation}

It follows that $\lambda_{1,1}$ has finite possibilities, and Claim
I-(3) is verified.
\end{proof}

\noindent\textbf{Claim I-(4)}. {\it The unit vector $X_1\in
U_pM_1^{n_1}$ given in Claim-I-(1) can be extended
differentiably to a unit vector field, still denoted by $X_1$, in a
neighbourhood $U\subset M^n$ of $p$, such that, for each
$q\in U$, the function $f_1$ defined on $U_qM_1^{n_1}$
attains its absolute maximum at $X_1(q)$.}
\begin{proof}[Proof of \textbf{Claim I-(4)}]
Let $\{E_1,\ldots,E_{n_1}\}$ be differentiable orthonormal vector
fields defined on a neighbourhood $U'$ of $p$ and satisfying
$E_i(q)\in T_qM_1^{n_1}, q\in U',1\le i\le n_1$, such that
$E_i(p)=X_i$ for $1\le i\le n_1$. Then, from the fact
$K_{X_1}X_1=\lambda_{1,1}X_1$ at $p$, we define a function $\gamma$
by
$$
\gamma:\mathbb{R}^{n_1}\times U'\rightarrow \mathbb{R}^{n_1}\ \ {\rm
by}\ \ (a_1,\ldots,a_{n_1},q)\mapsto(b_1,\ldots,b_{n_1}),
$$
where
\begin{equation}\label{eqn:3.49add1}
b_k=\sum_{i,j=1}^{n_1}a_ia_jh(K_{E_i}E_j,E_k)-\lambda_{1,1}a_k,\ \
k=1,2,\ldots,n_1,
\end{equation}
are regarded as functions on $\mathbb{R}^{n_1}\times U'$:
$b_k=b_k(a_1,\ldots,a_{n_1},q)$. Here, according to
\eqref{eqn:3.34} and the proof of Claim I-(2), the maximum of
$f_1$ defined on $U_qM_1^{n_1}$ is independent of $q\in
U'$, and it is equal to $\lambda_{1,1}=(n-1)\sqrt{-H/n}$.

Using \eqref{eqn:3.28add} and the fact that $f_1$ attains the
absolute maximum $\lambda_{1,1}$ at $E_1(p)$, we obtain that
\begin{equation*}
\begin{aligned}
\tfrac{\partial b_k}{\partial a_m}(1,0,\ldots,0,p)&=2h(K_{E_1(p)}E_m(p),E_k(p))-\lambda_{1,1}\delta_{km}\\
&=\left\{
\begin{aligned}
&0,\ \ {\rm if}\ k\neq m, \\
&\lambda_{1,1},\ \ {\rm if}\ k=m=1,\\
&2\lambda_{1,k}-\lambda_{1,1},\ \ {\rm if}\ k=m\ge2.
\end{aligned}
\right.
\end{aligned}
\end{equation*}
From the proof of Claim-I-(1) we have $\lambda_{1,1}>0$ and
$\lambda_{1,1}>2\lambda_{1,k}$ for $2\le k\le n_1$. Then, the implicit
function theorem shows that there exist differentiable functions
$\{a_i(q)\}_{1\le i\le n_1}$, defined on a neighbourhood
$U''\subset U'$ of $p$, such that
\begin{equation}\label{eqn:3.49add2}
\left\{
\begin{aligned}
&a_1(p)=1,\ a_2(p)=\cdots=a_{n_1}(p)=0,\\
&b_i(a_1(q),\ldots,a_{n_1}(q),q)\equiv0,\ \ 1\le i\le n_1.
\end{aligned}
\right.
\end{equation}

Define the local vector field $V$ on $U''$ by
$$
V(q)=a_1(q)E_1(q)+\cdots+a_{n_1}(q)E_{n_1}(q),\ \
q\in U''.
$$
Then, from \eqref{eqn:3.49add1}, \eqref{eqn:3.49add2} and
\eqref{eqn:3.2}, we get
\begin{equation}\label{eqn:3.49add3}
K_VV=\lambda_{1,1}V.
\end{equation}

Let us define $\|V\|=\sqrt{h(V,V)}$. Since $\|V\|(p)=1$, there
exists a neighbourhood $U\subset U''$ of $p$, such that $V\not=0$
on $U$, and it holds that
$$
K_{\tfrac{V}{\sqrt{h(V,V)}}}\tfrac{V}{\sqrt{h(V,V)}}
=\tfrac{\lambda_{1,1}}{\sqrt{h(V,V)}}\tfrac{V}{\sqrt{h(V,V)}}.
$$

From Claim I-(3), we know that
$\tfrac{\lambda_{1,1}}{\sqrt{h(V,V)}}$ takes values of finite
number. On the other hand, $\tfrac{\lambda_{1,1}}{\sqrt{h(V,V)}}$ is
continuous and $h(V,V)(p)=1$. Thus $h(V,V)\equiv1$.
It follows from \eqref{eqn:3.49add3} that, for any point $q\in
U$, the function $f_1$ attains its absolute maximum in
$V(q)$.

Define $X_1:=V$ on $U$. Then we have completed the proof of Claim I-(4).
\end{proof}

Finally, having determined the unit vector field $X_1$ as in Claim
I-(4), we can further choose orthonormal vectors
$X_2,\ldots,X_{n_1}$ orthogonal to $X_1$, defined on
$U$ and satisfying $X_i(q)\in T_qM_1^{n_1},
q\in U,2\le i\le n_1$. Then, it is easily seen
that, combining with Lemma \ref{lm:3.2}, Claim I-(1), Claim I-(2)
and their proofs, $\{X_1,\ldots,X_{n_1}\}$ turns into the desired
local orthonormal vector fields so that we have completed
the proof for the first step of induction.

\vskip 2mm

\noindent{\bf The second step of induction}

In this step, we first assume the assertion of Lemma \ref{lm:3.5}
for all $1\le i\le k$, where $k\in \{1,2,\ldots,n_1-2\}$ is a fixed
integer. Thus, we have:

{\it Around any given $p\in M_1^{n_1}\times M_2^{n_2}$,
there exist local orthonormal vector fields $\{X_i\}_{1\le i\le n_1}$ defined on
a neighborhood $U$ of $p$ and satisfying $X_i(q)\in T_qM_1^{n_1},
q\in U,1\le i\le n_1$, such that the difference tensor $K$ takes the form:
\begin{equation}\label{eqn:3.50}
\left\{
\begin{aligned}
&K_{X_1}X_1=\lambda_{1,1} X_1,\\
&K_{X_i}X_i=\mu_1 X_1+\cdots+\mu_{i-1} X_{i-1}+\lambda_{i,i}X_i,\ \ 2\le i\le k,\\
&K_{X_i}X_j=\mu_i X_j,\ \ 1\le i\le k, \ i<j\le n_1,\\
&K_{X_i}Y=\mu_iY,\ Y(q)\in T_qM_2^{n_2},\ \ 1\le i\le k,
\end{aligned}
\right.
\end{equation}
where, $\mu_i$ and $\lambda_{i,i}$ for $1\le i\le k$ are real
numbers, and they satisfy the relations:
\begin{equation}\label{eqn:3.51}
\lambda_{i,i}+(n-i)\mu_i=0,\ \lambda_{i,i}>0,\ \ 1\le i\le k.
\end{equation}
Moreover, at any $q\in U$, the number $\lambda_{i,i}$ is
the maximum of the function $f_1$ defined on
$$
\left\{u\in U_qM_1^{n_1}\,|\,u\perp X_1(q),\ldots,\,u\perp
X_{i-1}(q)\right\},
$$
for each $1\le i\le k$.}

\vskip 2mm

Then, as purpose of the second step, we should verify the
assertion of Lemma \ref{lm:3.5} for $i=k+1$. To do so, we are
sufficient to show that:

{\it There exist an orthonormal frame $\{\tilde{X}_i\}_{1\le i\le
n_1}$ on $TM_1^{n_1}$ around $p$, given by
$$
\tilde{X}_1=X_1,\ldots,\tilde{X}_k=X_k;\
\tilde{X}_l=\sum_{t=k+1}^{n_1}T^t_lX_t,\ \ k+1\le l\le n_1,
$$
such that $T=(T_l^t)_{k+1\le l,t\le n_1}$ is an orthogonal matrix,
and the difference tensor $K$ takes the following form:
\begin{equation}\label{eqn:3.52}
\left\{
\begin{aligned}
&K_{\tilde{X}_1}\tilde{X}_1=\lambda_{1,1} \tilde{X}_1,\\
&K_{\tilde{X}_i}\tilde{X}_i=\mu_1 \tilde{X}_1+\cdots+\mu_{i-1}\tilde{X}_{i-1}+\lambda_{i,i}\tilde{X}_i,\ \ 2\le i\le k+1,\\
&K_{\tilde{X}_i}\tilde{X}_j=\mu_i \tilde{X}_j,\ \ 1\le i\le k+1, \ i+1\le j\le n_1,\\
&K_{\tilde{X}_i}Y=\mu_iY,\ \ Y(q)\in T_qM_2^{n_2},\ \ 1\le i\le k+1,
\end{aligned}
\right.
\end{equation}
where, $\mu_i$ and $\lambda_{i,i}$, for $1\le i\le k+1$, are real
numbers, and they satisfy the relations
\begin{equation}\label{eqn:3.53}
\lambda_{i,i}+(n-i)\mu_i=0,\ \ 1\le i\le k+1.
\end{equation}
Moreover, at any $q$ around $p$, the number $\lambda_{i,i}$ is
the maximum of the function $f_1$ defined on
$$
\big\{u\in U_qM_1^{n_1}\,|\, u\perp{\rm span}\{\tilde
X_1(q),\ldots,\tilde X_{i-1}(q)\}\big\},
$$
for each $1\le i\le k+1$.}

\vskip 2mm

In order to prove the above conclusions, similar to the proof in the
first step, we also divide it into the verification of the following
four claims.

\vskip 1mm

\noindent\textbf{Claim II-(1)}. {\it For any $p\in M_1^{n_1}\times
M_2^{n_2}$, there exist an orthonormal basis $\{\bar{X}_i\}_{1\le
i\le n_1}$ of $T_pM_1^{n_1}$ and, real numbers
$\lambda_{k+1,k+1}>0$, $\lambda_{k+1,k+2}=\cdots=\lambda_{k+1,n_1}$
and $\mu_{k+1}$, such that the following relations hold:}
\begin{equation}\label{eqn:3.53add}
\left\{
\begin{aligned}
&K_{\bar{X}_1}\bar{X}_1=\lambda_{1,1}\bar{X}_1,\\
&K_{\bar{X}_i}\bar{X}_i=\mu_1\bar{X}_1+\cdots+\mu_{i-1}\bar{X}_{i-1}+\lambda_{i,i}X_i,\ \ 2\le i\le k+1,\\
&K_{\bar{X}_i}\bar{X}_j=\mu_i\bar{X}_j,\ \ 1\le i\le k,\ \ i+1\le j\le n_1,\\
&K_{\bar{X}_{k+1}}\bar{X}_i=\lambda_{k+1,i}\bar{X}_i,\,\ \ k+2\le i\le n_1,\\
&K_{\bar{X}_{k+1}}Y=\mu_{k+1}Y,\ \ Y\in T_pM_2^{n_2}.
\end{aligned}
\right.
\end{equation}

\begin{proof}[Proof of \textbf{Claim II-(1)}]

By the assumption of induction, we have local orthonormal vector
fields $\{X_i\}_{1\le i\le n_1}$ defined on a neighborhood $U$ of
$p$ and satisfying $X_i(q)\in T_qM_1^{n_1}$ for $q\in U$ and $1\le
i\le n_1$, such that \eqref{eqn:3.50} and \eqref{eqn:3.51} hold. We
first take $\bar{X}_1=X_1(p),\ldots,\bar{X}_k=X_k(p)$ and put
$$
V_k=\{u\in T_{p}M_1^{n_1} \mid u\perp \bar{X}_1,\dots,u\perp
\bar{X}_k\}.
$$

Then, similar argument as in the proof of Claim I-(1) shows that
when restricting on $V_k\cap U_pM_1^{n_1}$ the function
$f_1\not=0$. Thus, we can choose a unit vector $\bar{X}_{k+1}\in
V_k$ such that
$\lambda_{k+1,k+1}=h(K_{\bar{X}_{k+1}}\bar{X}_{k+1},\bar{X}_{k+1})$
is the maximum of $f_1$ on $V_k\cap U_pM_1^{n_1}$ with
$\lambda_{k+1,k+1}>0$.

Define a linear transformation $\mathfrak{A}:V_k\to V_k$ by
$$
\mathfrak{A}(X)=K_{\bar{X}_{k+1}}X-\sum_{i=1}^kh(K_{\bar{X}_{k+1}}X,
\bar{X}_i)\bar{X}_i,\ \ \forall\, X\in V_k.
$$

It is easily seen that $\mathfrak{A}$ is self-adjoint and
$\mathfrak{A}(\bar{X}_{k+1})=\lambda_{k+1,k+1}\bar{X}_{k+1}$. We can
choose orthonormal vectors $\bar{X}_{k+2},\ldots,\bar{X}_{n_1}\in
V_k$ orthogonal to $\bar{X}_{k+1}$, which are the remaining
eigenvectors of $\mathfrak{A}$ with associated eigenvalues
$\lambda_{k+1,k+2}, \ldots, \lambda_{k+1,n_1}$, respectively. Then,
by the assumption \eqref{eqn:3.50} of induction, we can show that
\begin{equation}\label{eqn:3.54}
\left\{
\begin{aligned}
&K_{\bar{X}_{k+1}}\bar{X}_{k+1}=\mu_1\bar{X}_1+\cdots+\mu_{k-1}\bar{X}_{k-1}+\mu_k\bar{X}_k
+\lambda_{k+1,k+1}\bar{X}_{k+1},\\&
K_{\bar{X}_{k+1}}\bar{X}_i=\lambda_{k+1,i}\bar{X}_i,\ \ k+2\le
i\le n_1.
\end{aligned}
\right.
\end{equation}

Taking $X=Z=\bar{X}_{k+1}$ and $Y=\bar{X}_i$ in \eqref{eqn:2.3} for
$k+2\le i\le n_1$, using \eqref{eqn:2.10} and \eqref{eqn:3.54}, we
can obtain
\begin{equation}\label{eqn:3.55}
\lambda_{k+1,i}^2-\lambda_{k+1,k+1}\lambda_{k+1,i}+H-\sum_{l=1}^k\mu_l^2=0,\
\ k+2\le i\le n_1.
\end{equation}

Similar to the proof of \eqref{eqn:3.13}, we have
$\lambda_{k+1,k+1}\ge 2\lambda_{k+1,i}$ for $k+2\le i\le n_1$. Then,
solving \eqref{eqn:3.55}, we get
$\lambda_{k+1,k+2}=\cdots=\lambda_{k+1,n_1}$ with
\begin{equation}\label{eqn:3.56}
\lambda_{k+1,i}=\tfrac{1}{2}\Big(\lambda_{k+1,k+1}-\Big[\lambda_{k+1,k+1}^2
-4(H-\sum_{l=1}^k\mu_l^2)\Big]^{1/2}\Big),\ \ k+2\le i\le n_1.
\end{equation}

Similarly, taking in \eqref{eqn:2.3} $X=Z=X_{k+1}$ and $Y\in
T_pM_2^{n_2}$ a unit vector, then using \eqref{eqn:2.10} and
\eqref{eqn:3.2}, we get
\begin{equation}\label{eqn:3.57}
\mu_{k+1}^2-{\mu_{k+1}}\lambda_{k+1,k+1}+H-\sum_{l=1}^k\mu_l^2=0.
\end{equation}

Hence, we have
\begin{equation}\label{eqn:3.58}
\mu_{k+1}=\tfrac{1}{2}\Big(\lambda_{k+1,k+1}+\varepsilon_{k+1}\Big[\lambda_{k+1,k+1}^2
-4(H-\sum_{l=1}^k\mu_l^2)\Big]^{1/2}\Big),\ \
\varepsilon_{k+1}=\pm1.
\end{equation}

On the other hand, by applying ${\rm trace}\,K_{\bar{X}_{k+1}}=0$,
we get $n_1-n_2\varepsilon_{k+1}-k-1>0$ and that
\begin{equation}\label{eqn:3.59}
\lambda_{k+1,k+1}=2(n_1-n_2\varepsilon_{k+1}-k-1)\sqrt{\tfrac{\sum_{l=1}^k\mu_l^2-H}
{(n_1+n_2-k+1)^2-(n_1-n_2\varepsilon_{k+1}-k-1)^2}}.
\end{equation}

From \eqref{eqn:3.56}, \eqref{eqn:3.58}, \eqref{eqn:3.59} and the
assumption that $\mu_1,\ldots,\mu_k$ are real numbers, we see that,
as claimed, $\lambda_{k+1,k+2}=\cdots=\lambda_{k+1,n_1}$ and
$\mu_{k+1}$ are also constants.

Moreover, by \eqref{eqn:3.54} and the assumption \eqref{eqn:3.50} of
induction, we get the assertion that \eqref{eqn:3.53add} holds.
\end{proof}

\noindent\textbf{Claim II-(2)}. {\it The real numbers described in
Claim II-(1) satisfy the relations:}
$$
\lambda_{k+1,k+2}=\cdots=\lambda_{k+1,n_1}=\mu_{k+1}\ \ {\rm and}\ \
\lambda_{k+1,k+1}+(n-k-1)\mu_{k+1}=0.
$$

\begin{proof}[Proof of \textbf{Claim II-(2)}]

From \eqref{eqn:3.56} and \eqref{eqn:3.58}, the first assertion is
equivalent to showing that $\varepsilon_{k+1}=-1$. Suppose on the
contrary that $\varepsilon_{k+1}=1$. Then we have
\begin{equation}\label{eqn:3.60}
\mu_{k+1}\lambda_{k+1,i}=H-\sum_{l=1}^k\mu_l^2,\ \ k+2\le i\le n_1.
\end{equation}

Now from ${\rm trace}\,K_{\bar{X}_{k+1}}=0$ and
$\lambda_{k+1,k+1}>0$ we obtain
\begin{equation}\label{eqn:3.61}
n_1-n_2-k-1>0,
\end{equation}
and that
\begin{equation}\label{eqn:3.62}
\lambda_{k+1,k+1}=2(n_1-n_2-k-1)\sqrt{\tfrac{\sum_{l=1}^k\mu_l^2-H}
{(n_1+n_2-k+1)^2-(n_1-n_2-k-1)^2}}.
\end{equation}

Put $V_{k+1}=\{u\in T_pM_1^{n_1} \mid u\perp
\bar{X}_1,\dots,u\perp \bar{X}_{k+1}\}$. Then \eqref{eqn:3.61} shows
that $\dim V_{k+1}=n_1-k-1\ge n_2+1\ge3$. Again, similar argument as
in the proof of Claim I-(1) shows that, restricting on $V_{k+1}\cap
U_pM_1^{n_1}$, the function $f_1\not=0$.

Now, by a totally similar argument as in the proof of Claim II-(1),
we can choose a new orthonormal basis $\{\tilde X_i\}_{1\le i\le
n_1}$ of $T_pM_1^{n_1}$ with $\tilde X_j=\bar{X}_j$ for $1\le
j\le k+1$, such that $f_1$, restricting on $V_{k+1}\cap
U_pM_1^{n_1}$, attains its maximum $\lambda_{k+2,k+2}>0$ at
$\tilde X_{k+2}$ so that $\lambda_{k+2,k+2} =h(K_{\tilde
X_{k+2}}\tilde X_{k+2},\tilde X_{k+2})$.

Similar as before, we define a self-adjoint operator $\mathfrak{B}:\
V_{k+1}\to V_{k+1}$ by
\begin{equation*}
\mathfrak{B}(X)=K_{\tilde X_{k+2}}X- \sum_{i=1}^{k+1}h(K_{\tilde
X_{k+2}}X,\tilde X_i)\tilde X_i.
\end{equation*}

Then $\mathfrak{B}(\tilde X_{k+2})=\lambda_{k+2,k+2}\tilde X_{k+2}$.
As before we can choose orthonormal vectors $\tilde
X_{k+3},\ldots,\tilde X_{n_1}\in V_{k+1}$, orthogonal to $\tilde
X_{k+2}$, which are the remaining eigenvectors of $\mathfrak{B}:\
V_{k+1}\to V_{k+1}$, with associated eigenvalues $\lambda_{k+2,k+3},
\ldots, \lambda_{k+2,n_1}$, respectively.

In this way, by using \eqref{eqn:3.53add}, we can show that
\begin{equation}\label{eqn:3.63}
\left\{
\begin{aligned}
&K_{\tilde X_{k+2}}\tilde X_{k+2}=\mu_1 \tilde
X_1+\cdots+\mu_{k}\tilde X_{k}+\lambda_{k+1,k+2}\tilde
X_{k+1}+\lambda_{k+2,k+2}\tilde
X_{k+2},\\
&K_{\tilde X_{k+2}}\tilde X_i=\lambda_{k+2,i}\tilde X_i,\ \ k+3\le
i\le n_1.
\end{aligned}
\right.
\end{equation}

Taking in \eqref{eqn:2.3} that $X=Z=\tilde X_{k+2}$ and $Y=\tilde
X_i$ for $k+3\le i\le n_1$, and using \eqref{eqn:2.10}, we can
obtain
\begin{equation}\label{eqn:3.64}
\lambda_{k+2,i}^2-\lambda_{k+2,k+2}\lambda_{k+2,i}
+H-\sum_{l=1}^k\mu_l^2-\lambda_{k+1,i}^2=0,\ \ k+3\le i\le n_1.
\end{equation}

Notice that $\lambda_{k+2,k+2}\ge 2\lambda_{k+2,i}$ for $k+3\le
i\le n_1$. Then, solving \eqref{eqn:3.64}, we get
\begin{equation}\label{eqn:3.65}
\begin{aligned}
\lambda_{k+2,i}=\tfrac{1}{2}\Big(\lambda_{k+2,k+2}-\Big[\lambda_{k+2,k+2}^2
-4(H-\sum_{l=1}^k\mu_l^2-\lambda_{k+1,i}^2)\Big]^{1/2}\Big)
\end{aligned}
\end{equation}
for $k+3\le i\le n_1$. Thus,
$\lambda_{k+2,k+3}=\cdots=\lambda_{k+2,n_1}$.

On the other hand, taking  in \eqref{eqn:2.3} $X=Z=\tilde X_{k+2}$
and $Y\in T_pM_2^{n_2}$ a unit vector, then using \eqref{eqn:2.10} and
\eqref{eqn:3.2}, we get
\begin{equation}\label{eqn:3.66}
\mu_{k+2}^2-{\mu_{k+2}}\lambda_{k+2,k+2}+H
-\lambda_{k+1,i}\mu_{k+1}-\sum_{l=1}^k\mu_l^2=0,\ \ k+2\le i\le n_1.
\end{equation}

From \eqref{eqn:3.60} and \eqref{eqn:3.66}, we get
\begin{equation}\label{eqn:3.67}
{\mu^2_{k+2}}-{\mu_{k+2}}\lambda_{k+2,k+2}=0,
\end{equation}
and, equivalently,
\begin{equation}\label{eqn:3.68}
\mu_{k+2}=\tfrac{1}{2}(\lambda_{k+2,k+2}+\varepsilon_{k+2}\lambda_{k+2,k+2}),\
\ \varepsilon_{k+2}=\pm1.
\end{equation}

Then, from ${\rm trace}\,K_{\tilde X_{k+2}}=0$ and
$\lambda_{k+2,k+2}>0$, we get $n_1-k-2>0$ and that
\begin{equation}\label{eqn:3.69}
\lambda_{k+2,k+2}=2(n_1-k-2)\sqrt{\tfrac{\lambda_{k+1,k+2}^2
+\sum_{l=1}^k\mu_l^2-H}{(n_1+n_2-k+\varepsilon_{k+2}n_2)^2-(n_1-k-2)^2}}\,.
\end{equation}

From \eqref{eqn:3.61} and that $n_2\ge2$, we have the following
calculations
\begin{equation}\label{eqn:3.70}
\begin{aligned}
\tfrac{n_1+n_2-k+1}{n_1-n_2-k-1}-\tfrac{n_1+n_2+\varepsilon_{k+2}n_2-k}{n_1-k-2}&>\tfrac{n_1+n_2-k+1}{n_1-n_2-k-1}-\tfrac{n_1+2n_2-k}{n_1-k-2}\\
&=\tfrac{2(n_2+1)(n_2-1)}{(n_1-n_2-k-1)(n_1-k-2)}>0.
\end{aligned}
\end{equation}

Then, by \eqref{eqn:3.62} and \eqref{eqn:3.69}, we get
$\lambda_{k+2,k+2}>\lambda_{k+1,k+1}$, which is a contradiction.

Hence, as claimed we have $\varepsilon_{k+1}=-1$ and
$\lambda_{k+1,k+2}=\cdots=\lambda_{k+1,n_1}=\mu_{k+1}$.

Finally, by ${\rm trace}\,K_{\bar{X}_{k+1}}=0$ and
\eqref{eqn:3.53add}, we get the second assertion that
$$
\lambda_{k+1,k+1}+(n-k-1)\mu_{k+1}=0.
$$

This completes the proof of Claim II-(2).
\end{proof}

\vskip 1mm

\noindent\textbf{Claim II-(3)}. {\it Under the assumptions of
induction, the set
$$
\Omega_{p,k}:=\Big\{\lambda\in\mathbb{R}\,|V\in
U_pM_1^{n_1}\setminus{\rm span}\{X_i(p)\}_{i=1}^k\ {\rm s.\,t.}\
K_VV=\lambda V+\sum_{i=1}^k\mu_iX_i\Big\}
$$
consists of finite numbers, which are independent of $p\in M^n$.}
\begin{proof}[Proof of \textbf{Claim II-(3)}]

We first notice that, for any fixed $p\in M^n$, Claim II-(1) shows
that $\lambda_{k+1,k+1}\in\Omega_{p,k}$ with $V=\bar X_{k+1}$. Thus,
the set $\Omega_{p,k}$ is non-empty.

Next, with the local orthonormal vector fields $\{X_i\}_{1\le i\le
n_1}$ around $p\in M_1^{n_1}\times M_2^{n_2}$, given by the
assumption of induction, we assume an arbitrary
$\lambda\in\Omega_{p,k}$ associated with $V\in
U_pM_1^{n_1}\setminus{\rm span}\{X_i(p)\}_{i=1}^k$ such that
$$
K_VV=\lambda V+\mu_1X_1+\cdots+\mu_kX_k.
$$

Then, at $p$, we put $\tilde X_{k+1}:=V$, $\tilde X_i=X_i(p)$
for $1\le i\le k$ and $\tilde\lambda_{k+1,k+1}=\lambda$.

Put $W_k=\{u\in T_pM_1^{n_1} \mid u\perp \tilde
X_1,\ldots,u\perp \tilde X_k\}$ and define $\mathfrak{F}: W_k\to
W_k$ by
$$
\mathfrak{F}(X)=K_{\tilde X_{k+1}}X-\sum_{i=1}^kh(K_{\tilde
X_{k+1}}X, \tilde X_i)\tilde X_i,\ \ X\in W_k.
$$

Then $\mathfrak{F}$ is a self-adjoint linear transformation and that
$\mathfrak{F}(\tilde X_{k+1})=\tilde\lambda_{k+1,k+1} \tilde
X_{k+1}$. Thus, we can choose an orthonormal basis $\{\tilde
X_i\}_{k+1\le i\le n_1}$ of $W_k$, such that
$$
\mathfrak{F}(\tilde X_i)=\tilde\lambda_{k+1,i}\tilde X_i,\ \ k+2\le
i\le n_1.
$$

Then, just like having did with equation \eqref{eqn:3.55}, we have
an integer $n_{1,k+1}$ with $0\le n_{1,k+1}\le n_1-(k+1)$ such that,
if necessary after renumbering the basis of $W_k$, it holds
\begin{equation}\label{eqn:3.71}
\left\{
\begin{aligned}
&\tilde \lambda_{k+1,k+2}=\cdots=\tilde
\lambda_{k+1,n_{1,k+1}+k+1}\\[-2mm]
&\ \ \ \ \ \ \ \ \ \ \ \ =\frac{1}{2}\Big(\tilde
\lambda_{k+1,k+1}+\Big[\tilde \lambda_{k+1,k+1}^2
        -4(H-\sum_{l=1}^k\mu_l^2)\Big]^{1/2}\Big),\\[-2mm]
&\tilde \lambda_{k+1, n_{1,k+1}+k+2}=\cdots=\tilde
\lambda_{k+1,n_1}\\[-2mm]
&\ \ \ \ \ \ \ \ \ \ \ \ =\frac{1}{2}\Big(\tilde
\lambda_{k+1,k+1}-\Big[\tilde \lambda_{k+1,k+1}^2
        -4(H-\sum_{l=1}^k\mu_l^2)\Big]^{1/2}\Big).
\end{aligned}\right.
\end{equation}

Similar as deriving \eqref{eqn:3.58}, now we also have
\begin{equation}\label{eqn:3.71add}
\mu_{k+1}=\tfrac{1}{2}\Big(\tilde
\lambda_{k+1,k+1}+\varepsilon_{k+1}\Big[\tilde \lambda_{k+1,k+1}^2
-4(H-\sum_{l=1}^k\mu_l^2)\Big]^{1/2}\Big),\ \
\varepsilon_{k+1}=\pm1.
\end{equation}.

Then, computing ${\rm trace}\,K_{\tilde X_{k+1}}=0$, gives that
\begin{equation}\label{eqn:3.72}
\begin{aligned}
&(n_1+n_2-k+1)\tilde \lambda_{k+1,k+1}\\[-3mm]
&+(2n_{1,k+1}-n_1+n_2\varepsilon_{k+1}+k+1)\Big[\tilde
\lambda_{k+1,k+1}^2
        -4(H-\sum_{l=1}^k\mu_l^2)\Big]^{1/2}=0.
\end{aligned}
\end{equation}

From \eqref{eqn:3.72} we have proved the assertion that
$\lambda=\tilde \lambda_{k+1,k+1}$ takes values of only finite
possibilities and they are independent of the point $p$.
\end{proof}

\noindent\textbf{Claim II-(4)}. {\it Under the assumptions of
induction, the unit vector $\bar{X}_{k+1}\in T_pM_1^{n_1}$,
determined by Claim II-(1), can be extended differentiably to a
local unit vector field in a neighbourhood $U$ of $p$, denoted by
$\tilde{X}_{k+1}$, such that for each $q\in U$ the function $f_1$,
defined on
$$
\mathcal{U}_k(q)=\{u\in U_qM_1^{n_1}\,|\,u\perp X_1(q),\dots,u\perp
X_k(q)\},
$$
attains its absolute maximum at $\tilde{X}_{k+1}(q)$.}

\begin{proof}[Proof of \textbf{Claim II-(4)}] First of all, according to
\eqref{eqn:3.59} and the proof of Claim II-(2), we notice that for
any $q$ around $p$, the maximum of $f_1$ defined on
$\mathcal{U}_k(q)$ is independent of $q$, and it equals to
$\lambda_{k+1,k+1}=(n-k-1)\sqrt{(\sum_{l=1}^k\mu_l^2-H)/(n-k)}$.

Now, we choose arbitrary differentiable orthonormal vector fields
$\{E_{k+1},\ldots,E_{n_1}\}$, defined on a neighbourhood $U'$ of $p$
such that, for $k+1\le i\le n_1$ and $q\in U'$, we have
$E_i(p)=\bar{X}_i$ and $E_i(q)\in\mathcal{U}_k(q)$.

Next, we define a function $\gamma$ by
$$
\begin{aligned}
\gamma:\ \mathbb{R}^{n_1-k}\times U'&\rightarrow
\mathbb{R}^{n_1-k},\\
(a_{k+1},\ldots,a_{n_1},q)&\mapsto(b_{k+1},\ldots,b_{n_1}),
\end{aligned}
$$
where
\begin{equation}\label{eqn:3.72add1}
b_l=\sum_{i,j=k+1}^{n_1}a_ia_jh(K_{E_i}E_j,E_l)-\lambda_{k+1,k+1}a_l,\
\ k+1\le l\le n_1,
\end{equation}
are regarded as functions on $\mathbb{R}^{n_1-k}\times U':\
b_l=b_l(a_{k+1},\ldots,a_{n_1},q)$.

Using Claim II-(1), the fact that $f_1$ attains its absolute
maximum $\lambda_{k+1,k+1}$ at $E_{k+1}(p)$, and that
$$
h(K_{E_{k+1}}E_i,E_j)|_{p}=\lambda_{k+1,i}\delta_{ij},\ \ k+1\le
i,j\le n_1,
$$
where $\lambda_{k+1,i}$ is given by \eqref{eqn:3.56}, we then obtain
that
\begin{equation*}
\begin{aligned}
\tfrac{\partial b_l}{\partial a_m}(1,0,\ldots,0,p)&=2h(
K_{E_{k+1}(p)}
E_m(p),E_l(p))-\lambda_{k+1,k+1}\delta_{lm}\\
&=\left\{
\begin{aligned}
&0,\ \ {\rm if}\ l\neq m ,\\
&\lambda_{k+1,k+1},\ \ {\rm if}\ l= m=k+1 ,\\
&2\lambda_{k+1,l}-\lambda_{k+1,k+1},\ \ {\rm if}\ k+2\le l= m\le
n_1.
\end{aligned}
\right.
\end{aligned}
\end{equation*}

Given that $\lambda_{k+1,k+1}>0$ and
$\lambda_{k+1,k+1}-2\lambda_{k+1,l}>0$ for $k+2\le l\le n_1$, the
implicit function theorem shows that in a neighbourhood $U''\subset
U'$ of $p$ there exist differentiable functions $\{a_{k+1},\ldots,
a_{n_1}\}$ satisfying
\begin{equation}\label{eqn:3.72add2}
\left\{
\begin{aligned}
&a_{k+1}(p)=1,\ a_{k+2}(p)=\cdots=a_{n_1}(p)=0,\\
&b_l(a_{k+1}(q),\ldots,a_{n_1}(q),q)\equiv0,\ q\in U'',\ \ k+1\le
l\le n_1.
\end{aligned}
\right.
\end{equation}

Define a local vector field $V$ on $U''$ by
$$
V(q)=a_{k+1}(q)E_{k+1}(q)+\cdots+a_{n_1}(q)E_{n_1}(q),\
\ q\in U''.
$$
Then $V(p)=\bar X_{k+1}$, there exists a neighbourhood $U\subset
U''$ of $p$, such that $V\not=0$ on $U$. Using
\eqref{eqn:3.72add1}, \eqref{eqn:3.72add2} and \eqref{eqn:3.2}, we
easily see that
$$
K_{V}V=\lambda_{k+1, k+1}V+\mu_1h(V,V) X_1+\cdots+\mu_k h(V,V) X_k,
$$
or, equivalently,
\begin{equation}\label{eqn:3.73}
K_{\tfrac{V}{\sqrt{h(V,V) }}} \tfrac{V}{\sqrt{h(V,V) }}
=\tfrac{\lambda_{k+1,k+1}}{\sqrt{h(V,V)}}
\tfrac{V}{\sqrt{h(V,V)}}+\sum_{i=1}^k\mu_iX_i,\ \ {\rm in}\ U.
\end{equation}

Now, according to Claim II-(3), the function
$\tfrac{\lambda_{k+1,k+1}}{\sqrt{h(V,V)}}$ takes values of only
finite possibilities. On the other hand,
$\tfrac{\lambda_{k+1,k+1}}{\sqrt{h(V,V)}}$ is continuous and $h(V,V)
(p)=1$. Thus $h(V,V)|_U\equiv1$. Let $\tilde{X}_{k+1}:=V$. Then,
\eqref{eqn:3.73} with $h(V,V)=1$ implies that
$$
K_{\tilde{X}_{k+1}}\tilde{X}_{k+1}=\lambda_{k+1,
k+1}\tilde{X}_{k+1}+\mu_1 X_1+\cdots+\mu_k X_k,
$$
and for any $q\in U$, $f_1$ attains its absolute maximum
$\lambda_{k+1, k+1}$ at $\tilde{X}_{k+1}(q)$.
\end{proof}

Let $\tilde{X}_1=X_1,\ldots,\tilde{X}_k=X_k$ and choose vector
fields $\tilde{X}_{k+2},\ldots,\tilde{X}_{n_1}$ such that, with
$\tilde{X}_{k+1}$ obtained as in Claim II-(4),
$\{\tilde{X}_1,\tilde{X}_2,\ldots,\tilde{X}_{n_1}\}$ is a local
orthonormal frame of $TM_1^{n_1}$ defined on a neighborhood $U$ of
$p$ and satisfies $X_i(q)\in T_qM_1^{n_1}$ for $q\in U$ and $1\le
i\le n_1$. Then, with respect to $\{\tilde{X}_i\}_{1\le i\le n_1}$
and combining with Lemma \ref{lm:3.2}, we immediately fulfil the
second step of induction.

In this way, the method of induction allows us to obtain the desired
orthonormal vector fields $\{X_1,\ldots,X_{n_1-1}\}$ defined on a
neighborhood $U$ of $p$ and satisfying $X_i(q)\in T_qM_1^{n_1}$ for
$q\in U$ and $1\le i\le n_1-1$. Finally, we choose a unit vector
field $X_{n_1}$ that is orthogonal to $\{X_1,\ldots,X_{n_1-1}\}$ and
that satisfies $X_{n_1}(q)\in T_qM_1^{n_1}$, such that
$\lambda_{n_1,n_1}\ge0$ (if necessary we change $X_{n_1}$ by
$-X_{n_1}$). Then, it is easy to see that $\{X_1,\ldots,X_{n_1}\}$
are the desired orthonormal vector fields. Accordingly, we have
completed the proof of Lemma \ref{lm:3.5}.
\end{proof}

\section{Proofs of the Theorems and Corollaries}\label{sect:4}


First of all, continuing with the study of \textbf{Case
$\mathfrak{C}_1$} in last section, we show that the local
orthonormal vector fields $\{X_i\}_{1\le i\le n_1}$, as determined
in Lemma \ref{lm:3.5}, consist of parallel vector fields such that
$\hat{\nabla} X_i=0$.


\begin{lem}\label{lm:3.6}
The local orthonormal vector fields $\{X_1,\ldots,X_{n_1}\}$, as
described by Lemma \ref{lm:3.5}, consist of parallel vector fields,
i.e.,
\begin{equation*}
\hat{\nabla} X_i=0,\ \ 1\le i\le n_1.
\end{equation*}
\end{lem}
\begin{proof}
We shall give the proof by induction on $i$. First of all, we prove
$\hat{\nabla}X_1=0$.

In fact, for $j\ge2$, applying \eqref{eqn:3.27}, we have the
following calculations
\begin{equation}\label{eqn:3.74}
(\hat{\nabla}_{X_j}K)(X_1,X_1)=(\lambda_{1,1}-2\mu_1)\hat{\nabla}_{X_j}X_1,
\end{equation}
\vskip -5mm
\begin{equation}\label{eqn:3.75}
(\hat{\nabla}_{X_1}K)(X_j,X_1)=\mu_1\hat{\nabla}_{X_1}X_j-K(\hat{\nabla}_{X_1}X_j,X_1)-K(\hat{\nabla}_{X_1}X_1,X_j).
\end{equation}

Now, the Codazzi equation
$(\hat{\nabla}_{X_j}K)(X_1,X_1)=(\hat{\nabla}_{X_1}K)(X_j,X_1)$
gives that
\begin{equation}\label{eqn:3.76}
(\lambda_{1,1}-2\mu_1)\hat{\nabla}_{X_j}X_1=\mu_1\hat{\nabla}_{X_1}X_j
-K(\hat{\nabla}_{X_1}X_j,X_1)-K(\hat{\nabla}_{X_1}X_1,X_j).
\end{equation}

Then, taking the component of \eqref{eqn:3.76} in direction of $X_1$
for each $j\ge2$ and using the fact that $h(\hat{\nabla}_{X_1}
X_1,Y)=0$ for $Y(q)\in T_qM_2^{n_2}$, and \eqref{eqn:3.27} again, we
get $\hat{\nabla}_{X_1} X_1=0$. Substituting $\hat{\nabla}_{X_1}
X_1=0$ into \eqref{eqn:3.76}, and then taking its component in
direction of $X_j$, we get $h(\hat{\nabla}_{X_j} X_1,X_k)=0$ for
$2\le j,k\le n_1$. This, together with the fact that
$h(\hat{\nabla}_{X_j} X_1,Y)=0$ for $Y(q)\in T_qM_2^{n_2}$, implies
that
\begin{equation}\label{eqn:3.76*}
\hat{\nabla}_{X_j} X_1=0,\ \ 1\le j\le
n_1.
\end{equation}

Take a unit vector field $Y$ with $Y(q)\in T_qM_2^{n_2}$. By a
direct calculation of
$h((\hat{\nabla}_{Y}K)(X_1,X_i),X_1)=h((\hat{\nabla}_{X_1}K)(Y,X_i),X_1)$,
we obtain $h(\hat{\nabla}_{Y} X_1,X_i)=0$ for $2\le i\le n_1$. This,
together with $h(\hat{\nabla}_{Y} X_1,Y')=0$ for $Y'(q)\in
T_qM_2^{n_2}$, implies that
\begin{equation}\label{eqn:3.76**}
\hat{\nabla}_{Y} X_1=0.
\end{equation}

Combining \eqref{eqn:3.76*} and \eqref{eqn:3.76**}, we have proved
the assertion $\hat{\nabla} X_1=0$.

\vskip 2mm

{\it Next, by induction we show that if for any fixed $2\le i \le
n_1-1$ satisfying
\begin{equation}\label{eqn:3.77}
\hat{\nabla}X_k=0,\ \ k=1,\ldots, i-1,
\end{equation}
then it holds $\hat{\nabla} X_i=0$.}

\vskip 2mm

To state a proof of the above second step, we consider five cases
below:

\vskip 2mm

(i) By \eqref{eqn:3.77} and that $h(X_k,X_l)=\delta_{kl}$, we
get
\begin{equation}\label{eqn:3.77add}
\begin{aligned}
&h(\hat{\nabla}_{X_j}X_i,X_k)=-h(\hat{\nabla}_{X_j}X_k,X_i)=0,\ \
1\le j\le n_1,\ k\le i\\ &h(\hat{\nabla}_YX_i,X_k)=-h(\hat{\nabla}_YX_k,X_i)=0,\ \
1\le k\le i\le n_1,\ .
\end{aligned}
\end{equation}

\vskip 2mm

(ii) For $j\le i-1$, by using \eqref{eqn:3.27}, \eqref{eqn:3.77} and
\eqref{eqn:3.77add}, we can show that
\begin{equation}\label{eqn:3.78}
\begin{aligned}
(\hat{\nabla}_{X_j}K)(X_i,X_i)&=\hat{\nabla}_{X_j}K(X_i,X_i)-2K(\hat{\nabla}_{X_j}X_i,X_i)\\
&=(\lambda_{i,i}-2\mu_i)\hat{\nabla}_{X_j}X_i,
\end{aligned}
\end{equation}
\vskip -2mm
\begin{equation}\label{eqn:3.79}
\begin{aligned}
(\hat{\nabla}_{X_i}K)(X_j,X_i)&=\hat{\nabla}_{X_i}K(X_j,X_i)
-K(\hat{\nabla}_{X_i}X_j,X_i)-K(\hat{\nabla}_{X_i}X_i,X_j)\\
&=\mu_j\hat{\nabla}_{X_i}X_i-K(\hat{\nabla}_{X_i}X_i,X_j).
\end{aligned}
\end{equation}

Then, by
$(\hat{\nabla}_{X_j}K)(X_i,X_i)=(\hat{\nabla}_{X_i}K)(X_j,X_i)$, for
$k\ge i+1$ we obtain
$$
\begin{aligned}
(\lambda_{i,i}-2\mu_i)h(\hat{\nabla}_{X_j}X_i,X_k)
&=\mu_jh(\hat{\nabla}_{X_i}X_i,X_k)-h(K(\hat{\nabla}_{X_i}X_i,X_j),X_k)\\
&=\mu_jh(\hat{\nabla}_{X_i}X_i,X_k)-h(\hat{\nabla}_{X_i}X_i,K_{X_j}X_k)=0.
\end{aligned}
$$
It follows that
\begin{equation}\label{eqn:3.80}
h(\hat{\nabla}_{X_j}X_i,X_k)=0,\ \ j\le i-1,\ k\ge i+1.
\end{equation}

\vskip 2mm

(iii) Similar to the above case (ii), for $j\ge i+1$, we have
\begin{equation}\label{eqn:3.81}
\begin{aligned}
(\hat{\nabla}_{X_j}K)(X_i,X_i)&=\hat{\nabla}_{X_j}K(X_i,X_i)-2K(\hat{\nabla}_{X_j}X_i,X_i)\\
&=(\lambda_{i,i}-2\mu_i)\hat{\nabla}_{X_j}X_i,
\end{aligned}
\end{equation}
\vskip -3mm
\begin{equation}\label{eqn:3.82}
\begin{aligned}
(\hat{\nabla}_{X_i}K)(X_j,X_i)&=\hat{\nabla}_{X_i}K(X_j,X_i)
-K(\hat{\nabla}_{X_i}X_j,X_i)-K(\hat{\nabla}_{X_i}X_i,X_j)\\
&=\mu_i\hat{\nabla}_{X_i}X_j-K(\hat{\nabla}_{X_i}X_j,X_i)-K(\hat{\nabla}_{X_i}X_i,X_j).
\end{aligned}
\end{equation}
Then, taking the $X_i$-components of $(\hat{\nabla}_{X_j}K)(X_i,X_i)
=(\hat{\nabla}_{X_i}K)(X_j,X_i)$, with using \eqref{eqn:3.27} and
\eqref{eqn:3.77}, we obtain
$$
\begin{aligned}
0&=(\lambda_{i,i}-2\mu_i)h(\hat{\nabla}_{X_j}X_i,X_i)\\
&=\mu_ih(\hat{\nabla}_{X_i}X_j,X_i)-h(K(\hat{\nabla}_{X_i}X_j,X_i),X_i)-h(K(\hat{\nabla}_{X_i}X_i,X_j),X_i)\\
&=-\mu_ih(\hat{\nabla}_{X_i}X_i,X_j)-h(\hat{\nabla}_{X_i}X_j,K_{X_i}X_i)-h(\hat{\nabla}_{X_i}X_i,K_{X_i}X_j)\\
&=(\lambda_{i,i}-2\mu_i)h(\hat{\nabla}_{X_i}X_i,X_j).
\end{aligned}
$$
Hence, we obtain
\begin{equation}\label{eqn:3.83}
h(\hat{\nabla}_{X_i}X_i,X_j)=0,\ j\ge i+1.
\end{equation}

\vskip 2mm

(iv) By using
$(\hat{\nabla}_{X_j}K)(X_i,X_i)=(\hat{\nabla}_{X_i}K)(X_j,X_i)$ and
taking its $X_k$-components for $j,k\ge i+1$, then applying
\eqref{eqn:3.83} we obtain
$$
\begin{aligned}
(\lambda_{i,i}-2\mu_i)h(\hat{\nabla}_{X_j}X_i,X_k)&=\mu_ih(\hat{\nabla}_{X_i}X_j,X_k)-h(K(\hat{\nabla}_{X_i}X_j,X_i),X_k)\\
&=\mu_ih(\hat{\nabla}_{X_i}X_j,X_k)-h(\hat{\nabla}_{X_i}X_j,K_{X_i}X_k)\\&=0.
\end{aligned}
$$
\vskip -3mm
\begin{equation}\label{eqn:3.84}
h(\hat{\nabla}_{X_j}X_i,X_k)=0,\ \ j,k\ge i+1.
\end{equation}

\vskip 2mm

(v) If $Y$ is a unit vector field with $Y(q)\in T_qM_2^{n_2}$, by a
direct calculation of
$h((\hat{\nabla}_{Y}K)(X_i,X_k),X_i)=h((\hat{\nabla}_{X_i}K)(Y,X_k),X_i)$
for $i+1\leq k$, we obtain
\begin{equation}\label{eqn:3.84*}
h(\hat{\nabla}_{Y} X_i,X_k)=0,\ \ i+1\leq k.
\end{equation}

For $Y,Y'$ with $Y(q),Y'(q)\in T_qM_2^{n_2}$, we have
$h(\hat{\nabla}_{X_j} X_i,Y)=h(\hat{\nabla}_Y X_i,Y')=0$. Hence,
combining \eqref{eqn:3.77add}, \eqref{eqn:3.80} and
\eqref{eqn:3.83}--\eqref{eqn:3.84*}, we finally get
\begin{equation}\label{eqn:3.85}
\hat{\nabla}X_i=0.
\end{equation}

Therefore, by induction we have proved that
\begin{equation}\label{eqn:3.850}
\hat{\nabla}X_i=0,\ \ 1\le i\le n_1-1.
\end{equation}

\vskip 2mm

Finally, for vector fields $X,Y$ with $X(q)\in
T_qM_1^{n_1},\,Y(q)\in T_qM_2^{n_2}$ and $k\le n_1-1$, from
\eqref{eqn:3.850} it is easily seen the following
$$
\begin{aligned}
&h(\hat{\nabla}_{X}X_{n_1},X_k)=-h(\hat{\nabla}_{X}X_k,X_{n_1})=0,\ h(\hat{\nabla}_{X}X_{n_1},X_{n_1})=0,\\
&h(\hat{\nabla}_YX_{n_1},X_k)=-h(\hat{\nabla}_YX_k,X_{n_1})=0,\
h(\hat{\nabla}_YX_{n_1},X_{n_1})=0,
\end{aligned}
$$
so that it holds also $\hat{\nabla}X_{n_1}=0$.

We have completed the proof of Lemma \ref{lm:3.6}.
\end{proof}


Moreover, we have the following further conclusion.

\begin{lem}\label{lm:3.7}
Let $x:M^{n}\rightarrow\mathbb{R}^{n+1}$ be an $n$-dimensional
locally strongly convex affine hypersphere such that \textbf{Case
$\mathfrak{C}_1$} in section \ref{sect:3} occurs, then the
difference tensor is parallel, i.e., $\hat{\nabla}K=0$.
\end{lem}

\begin{proof}
Let $\{X_1,\ldots,X_{n_1}\}$ be the local orthonormal vector fields
as described by Lemma \ref{lm:3.5}. Then Lemma
\ref{lm:3.6} shows that
\begin{equation}\label{eqn:3.86}
\hat{\nabla}X_i=0,\ \ 1\le i\le n_1.
\end{equation}

On the other hand, as $(M^n, h)=M_1^{n_1}(c_1)\times
M_2^{n_2}(c_2)$, by Proposition 56 in p.89 of \cite {O}, we can
choose local orthonormal vector fields $\{Y_1,\ldots,Y_{n_2}\}$ with
$Y_i(q)\in T_qM_2^{n_2}$, such that
\begin{equation}\label{eqn:3.87}
\hat{\nabla}_{X_i}Y_\alpha=0,\ \ 1\le
i\le n_1,\ 1\le \alpha\le n_2.
\end{equation}

Then, using \eqref{eqn:3.86}, \eqref{eqn:3.87} and properties of the
difference tensor established by Lemmas \ref{lm:3.2}, \ref{lm:3.4}
and \ref{lm:3.5}, direct calculations immediately give the assertion
that $\hat{\nabla}K=0$.
\end{proof}

\begin{them}\label{thm:3.1}
Let $x:M^n\rightarrow\mathbb{R}^{n+1}$ be an $n$-dimensional locally
strongly convex affine hypersphere such that \textbf{Case
$\mathfrak{C}_1$} in section \ref{sect:3} occurs. Then
$x:M^n\rightarrow\mathbb{R}^{n+1}$ is locally affinely equivalent to
the Calabi composition
\begin{equation}\label{eqn:3.1}
(x_1\cdots
x_{n_1})^2(x_{n+1}^2-x_{n_1+1}^2-\cdots-x_{n}^2)^{n_2+1}=1,
\end{equation}
where $(x_1,\ldots, x_{n+1})$ are the standard coordinates of
$\mathbb{R}^{n+1}$.
\end{them}

\begin{proof}
By Lemma \ref{lm:3.5} and Lemma \ref{lm:3.7}, we can apply Theorem
4.1 of \cite{HLSV} with $X_1$ being regarded as $e_1$ there. Then
$x:M^n\rightarrow\mathbb{R}^{n+1}$ is a Calabi product of a point
$G_1$ with a hyperbolic affinesphere $G_2':\tilde
M_1^{n-1}\to\mathbb R^n$ with parallel cubic form and affine mean
curvature $H_2$, so that we have the decomposition
$M^n=I_1\times\tilde M_1^{n-1}$, $I\subset\mathbb R$, and the
parametrization
$$
x(s_1,\tilde p_1)=-\tfrac{\mu_1}{H^2+\mu_1^2}e^{-s_1}G_1
+\tfrac{1}{H^2+\mu_1^2}e^{s_1/n}G_2'(\tilde p_1),\ \ s_1\in I_1,\
\tilde p_1\in\tilde M_1^{n-1}.
$$
Moreover, the affine metric of $G_2':\tilde M_1^{n-1}\to\mathbb R^n$
is $(\mu^2_1-H)h|_{T\tilde M_1^{n-1}}$ (cf. \cite{HLSV}).

Notice that $T\tilde M_1^{n-1}={\rm
span}\{X_2,\ldots,X_{n_1},Y_1,\ldots,Y_{n_2}\}$. Let us denote by
$K^{1}$ the difference tensor of $G_2':\tilde M_1^{n-1}\to\mathbb
R^n$, then from the proof of Theorem 4.1 of \cite{HLSV} and Lemmas
\ref{lm:3.4} and \ref{lm:3.5}, we can derive that $K^{1}$ has the
expressions as follows:
\begin{equation}\label{eqn:3.88}
\left\{
\begin{aligned}
&K^1_{X_i}X_i=\mu_2 X_2+\cdots+\mu_{i-1} X_{i-1}+\lambda_{i,i}X_i,\ \ 2\le i\le n_1,\\
&K^1_{X_i}X_j=\mu_iX_j,\ \ 2\le i<j\le n_1,\\
&K^1_{X_i}Y_\alpha=\mu_iY_\alpha,\ \ 2\le i\le n_1,\ 2\le \alpha\le n_1,\\
&K^1_{Y_\alpha}Y_\beta=\delta_{\alpha\beta}(\mu_2X_2+\cdots+\mu_{n_1}X_{n_1}),\
\ 1\le \alpha,\beta\le n_2.
\end{aligned}
\right.
\end{equation}

Notice also that, up to scaling a constant multiple,
$\{X_2,\ldots,X_{n_1},Y_1,\ldots,Y_{n_2}\}$ are the orthomormal
basis of the affine metric of $G_2':\tilde M_1^{n-1}\to\mathbb R^n$.
Applying Theorem 4.1 in \cite{HLSV} once again by regarding $X_2$
as $e_1$ there, then $G_2':\tilde M_1^{n-1}\to\mathbb R^n$ is a
Calabi product of a point $G_2$ with a hyperbolic affinesphere
$G_3':\tilde M_2^{n-2}\to\mathbb R^{n-1}$ with parallel cubic form,
so that we have the decomposition $\tilde M_1^{n-1}=I_2\times\tilde
M_2^{n-2}$, $I_2\subset\mathbb R$, and the further parametrization
$$
\begin{aligned}
x(s_1,s_2,\tilde p_2)=&-\tfrac{\mu_1}{H^2+\mu_1^2}e^{-s_1}G_1
-\tfrac{1}{H^2+\mu_1^2}\tfrac{\mu_2}{H_2^2+\mu_2^2}e^{\tfrac{s_1}{n}-s_2}G_2\\
&+\tfrac{1}{H^2+\mu_1^2}\tfrac{1}{H_2^2+\mu_2^2}e^{\tfrac{s_1}{n}-\tfrac{s_2}{n-1}}G_3'(\tilde
p_2),\ \ (s_1,s_2)\in I_1\times I_2,\ \tilde p_2\in \tilde
M_2^{n-2}.
\end{aligned}
$$

Continuing in this way $n_1$ times, we finally see that
$M^n=M_1^{n_1}\times M_2^{n_2}$, with $M_1^{n_1}\cong I_1\times
I_2\times\cdots \times I_{n_1}$, and
$x:M^n\rightarrow\mathbb{R}^{n+1}$ has a parametrization
\begin{equation}\label{eqn:3.89}
\begin{aligned}
x(s_1,&\ldots,s_{n_1},p_2)=-\tfrac{\mu_1}{H^2+\mu_1^2}e^{-s_1}G_1-\tfrac{1}{H^2+\mu_1^2}\tfrac{\mu_2}{H_2^2+\mu_2^2}e^{\tfrac{s_1}{n_1+n_2}-s_2}G_2-\cdots\\
&-\tfrac{1}{H^2+\mu_1^2}\tfrac{1}{H_2^2+\mu_2^2}\cdots\tfrac{1}{H_{n_1-1}^2+\mu_{n_1-1}^2}\tfrac{\mu_{n_1}}{H_{n_1}^2+\mu_{n_1}^2}
e^{\tfrac{s_1}{n_1+n_2}+\cdots+\tfrac{s_{n_1-1}}{n_2+2}-s_{n_1}}G_{n_1}\\
&+\tfrac{1}{H^2+\mu_1^2}\tfrac{1}{H_2^2+\mu_2^2}\cdots\tfrac{1}{H_{n_1}^2+\mu_{n_1}^2}
e^{\tfrac{s_1}{n_1+n_2}+\cdots+\tfrac{s_{n_1}}{n_2+1}}G_{n_1+1}'(p_2),\
\ p_2\in M_2^{n_2},
\end{aligned}
\end{equation}
where, $(s_1,\ldots,s_{n_1})\in M_1^{n_1}$, $\{G_i\}_{1\le i\le n_1}$ are constant
vectors and $G_{n_1+1}':M_2^{n_2}\to\mathbb{R}^{n_2+1}$ is a hyperbolic
affine hypersphere with parallel cubic form.

Furthermore, from the above procedure of induction, it can be easily
seen that $G_{n_1+1}':M_2^{n_2}\to\mathbb{R}^{n_2+1}$ has vanishing
difference tensor. This implies that
$G_{n_1+1}':M_2^{n_2}\to\mathbb{R}^{n_2+1}$ is a hyperboloid.
Therefore, up to an affine transformation, there exist constant
vectors $G_{n_1+1},\ldots,G_{n+1}$ such that
\begin{equation}\label{eqn:3.89add}
G_{n_1+1}'=y_1G_{n_1+1}+y_2G_{n_1+2}+\cdots+y_{n_2+1}G_{n+1},
\end{equation}
where $y_1^2+\cdots+y_{n_2}^2-y_{n_2+1}^2=-1$.

Combining \eqref{eqn:3.89} and \eqref{eqn:3.89add}, we finally see
that, up to an affine transformation, $x:M^n=M_1^{n_1}\times
M_2^{n_2}\to\mathbb{R}^{n+1}$ can be written as
$$
\begin{aligned}
x&=(x_1,\ldots,x_{n_1},x_{n_1+1},\ldots,x_{n+1})\\
&=\Big(e^{-s_1},e^{\tfrac{s_1}{n}-s_2},\ldots,e^{\tfrac{s_1}{n_1+n_2}+\cdots+\tfrac{s_{n_1-1}}{n_2+2}-s_{n_1}},
e^{\tfrac{s_1}{n_1+n_2}+\cdots+\tfrac{s_{n_1}}{n_2+1}}(y_1,\ldots,y_{n_2+1})\Big).
\end{aligned}
$$

Hence, $x:M^n\to\mathbb{R}^{n+1}$ is affinely equivalent to the
affine hypersphere \eqref{eqn:3.1}.
\end{proof}

\vskip 2mm

Next, we consider \textbf{Case $\mathfrak{C}_2$} as stated in
section \ref{sect:3} such that $x:M^n\rightarrow\mathbb{R}^{n+1}$ is
an $n$-dimensional locally strongly convex affine hypersphere with
$(M^n, h)=M_1^{n_1}(c_1)\times M_2^{n_2}(c_2)$, $n_1\ge2,\ n_2\ge2$
and $c_1c_2\neq0$. Then, similar to that in Lemma \ref{lm:3.2} for
the proof of \eqref{eqn:3.7}, we can obtain the following result.

\begin{lem}\label{lm:3.8}
For $p\in M_1^{n_1}\times M_2^{n_2}$, let $\{X_i\}_{1\le i\le n_1}$
and $\{Y_j\}_{1\le j\le n_2}$ be orthonormal bases of $T_pM_1^{n_1}$
and $T_pM_2^{n_2}$, respectively. Then, in \textbf{Case
$\mathfrak{C}_2$}, the difference tensor satisfies
\begin{equation}\label{eqn:3.1300}
\left\{
\begin{aligned}
&h(K_{X_i}X_j,Y_\gamma)=0,\ \ 1\le
i,j\le n_1,\ \ 1\le \gamma\le n_2,\\
&h(K_{Y_\alpha}Y_\beta,X_k)=0,\ \ 1\le \alpha,\beta \le n_2,\ \ 1\le
k\le n_1.
\end{aligned}
\right.
\end{equation}
\end{lem}

Moreover, we have the following further conclusion.

\begin{them}\label{thm:4.2}
Let $x:M^n\rightarrow\mathbb{R}^{n+1}$ be a locally strongly convex
product affine hypersphere, then \textbf{Case $\mathfrak{C}_2$} in
section \ref{sect:3} does not occur.
\end{them}
\begin{proof}
If otherwise, we assume that \textbf{Case $\mathfrak{C}_2$} does
occur. Then, as by Lemma \ref{lm:3.1} the difference tensor $K$
vanishes nowhere, we may assume that for an arbitrary fixed $p\in
M^n=M_1^{n_1}\times M_2^{n_2}$ there exists $X\in T_pM_1^{n_1}$ such
that $K_{X}\neq0$. Now, similar to the proof for the first step of
induction in the proof of Lemma \ref{lm:3.5}, we can show that
around $p\in M^n$ there exist local orthonormal vector fields
$\{X_1,\ldots,X_{n_1}\}$ with $X_i(q)\in T_qM_1^{n_1},1\le i\le
n_1$, such that the difference tensor takes the form
\begin{equation}\label{eqn:3.91}
K_{X_1}X_1=\lambda_1X_1,\ K_{X_1}X_i=\lambda_2X_i,\ \ 2\le i\le n_1,
\end{equation}
where, $\lambda_1$ and $\lambda_2$ are real numbers with
$\lambda_1>0$ and $\lambda_1+(n-1)\lambda_2=0$. Then, similar to the
proof of \eqref{3.26}, we can show that $\hat{\nabla}_{X_i}X_1=0$
for $1\le i\le n_1$. It follows that $\hat{R}(X_1,X_2)X_1=0$, which
is a contradiction to that $c_1c_2\neq0$.
\end{proof}


\noindent{\it The Completion of Theorem \ref{thm:1.1}'s Proof.}

If $c_1=c_2=0$, it follows from \eqref{eqn:2.10} that $(M^n,h)$ is
flat. Then, according to the result of \cite{VLS}, we get the
assertion (i) of Theorem \ref{thm:1.1}.

If $c_1^2+c_2^2\neq0$, we have two cases: \textbf{Case
$\mathfrak{C}_1$} and \textbf{Case $\mathfrak{C}_2$}, as preceding
described.

If \textbf{Case $\mathfrak{C}_1$} occurs, then by Theorem
\ref{thm:3.1}, we obtain the hypersphere as stated in (ii) of
Theorem \ref{thm:1.1}. Moreover, according to Theorem \ref{thm:4.2},
\textbf{Case $\mathfrak{C}_2$} does not occur.

We have completed the proof of Theorem \ref{thm:1.1}. \qed

\vskip 3mm

Next, we come to give the proof of Theorem \ref{thm:1.2}. First of
all, similar to the proof of Lemma \ref{lm:3.1}, we can obtain the
following result.

\begin{lem}\label{lm:3.9}
Let $x:M^n\rightarrow\mathbb{R}^{n+1}\ (n\ge3)$ be a locally
strongly convex affine hypersphere such that $(M^n,h)$ is locally
isometric to a Riemannian product $\mathbb{R}\times M_2^{n-1}(c_2)$,
where $M_2^{n-1}(c_2)$ is an $(n-1)$-dimensional Riemannian manifold
with constant sectional curvature $c_2\neq 0$. Then the difference
tensor $K$ of $x:M^n\rightarrow\mathbb{R}^{n+1}$ vanishes nowhere.
\end{lem}

Next, similar to the proofs of \eqref{eqn:3.4} and \eqref{eqn:3.5},
we have

\begin{lem}\label{lm:3.10}
Let $x:M^n\rightarrow\mathbb{R}^{n+1}$ be a locally strongly convex
affine hypersphere as described in Lemma \ref{lm:3.9}. For $p\in M^n$, assume that
$\{Y_\alpha\}_{1\le\alpha\le n-1}$ is an orthonormal basis of
$T_pM_2^{n-1}$ and $X\in T_p\mathbb{R}$ is a unit vector, then we
have
\begin{equation}\label{eqn:3.92}
K_{X}Y_\alpha=\mu(X)Y_\alpha,\ \ 1\le \alpha\le n_2,
\end{equation}
where $\mu(X)=:\mu$ depends only on $X$.
\end{lem}

Now, we will prove a lemma which plays the same important role as
Lemma \ref{lm:3.5}.

\begin{lem}\label{lm:3.11}
Let $x:M^n\rightarrow\mathbb{R}^{n+1}$ be a locally strongly convex
affine hypersphere as described in Lemma \ref{lm:3.9} with
$S=H\,{\rm id}$. Then, around any point $p\in M^n$, there exists a
local orthonormal frame $\{X_1,Y_1,\ldots,Y_{n-1}\}$ on $M^n$ with
$X_1(q)\in T_q\mathbb{R}$ and $Y_\alpha(q)\in T_qM_2^{n-1},\
1\le\alpha\le n-1$, such that the difference tensor of
$x:M^n\rightarrow\mathbb{R}^{n+1}$ takes the following form
\begin{equation}\label{eqn:3.93}
\left\{
\begin{aligned}
&K_{X_1}X_1=(n-1)\sqrt{-\tfrac Hn}\, X_1,\ K_{X_1}Y_\alpha=-\sqrt{-\tfrac Hn}\, Y_\alpha,\ \ 1\le \alpha\le n-1,\\
&K_{Y_\alpha}Y_\beta=-\sqrt{-\tfrac Hn}\,\delta_{\alpha\beta} X_1,\
\ 1\le \alpha,\beta\le n-1.
\end{aligned}
\right.
\end{equation}
Moreover, we have $c_2=(n+1)H/n<0$ and $\hat{\nabla}K=0$.
\end{lem}
\begin{proof}
Around any point $p\in M^n=I\times M_2^{n-1}$, we take local unit
vector fields $X$ and $Y$, with $X(q)\in T_q\mathbb{R}$ and $Y(q)\in
T_qM_2^{n-1}$. Then, similar to the proof of \eqref{eqn:3.7}, and
applying Lemma \ref{lm:3.10}, we obtain
\begin{equation}\label{eqn:3.94}
K_{X}X=\lambda X,\ \ K_{X}Y=\mu Y.
\end{equation}

Moreover, by using \eqref{eqn:2.3} and the fact $\hat R(Y,X)X=0$, we
have
\begin{equation}\label{eqn:3.95}
\mu^2-\lambda\mu+H=0.
\end{equation}

On the other hand, by ${\rm trace}\,K_{X}=0$, we get
$\lambda+(n-1)\mu=0$. This together with \eqref{eqn:3.95} implies
that, if necessary replacing $X$ by $-X$,
\begin{equation}\label{eqn:3.96}
H\le0,\ \ \lambda=(n-1)\sqrt{-\tfrac{H}{n}},\ \
\mu=-\sqrt{-\tfrac{H}{n}}.
\end{equation}

Similar to the proof of Lemma \ref{lm:3.4}, we can also show that
$$
h(K_YY',Y'')=0,\ \ \forall\, Y,Y',Y''\in T_qM_2^{n_2}.
$$

Since $(M^n,h)=\mathbb{R}\times M_2^{n-1}(c_2)$, by Proposition 56
in p.89 of \cite {O}, we can take an orthonormal frame
$\{X_1,Y_1,\ldots,Y_{n-1}\}$ on $M^{n}$ with $X_1=X$, such that
\begin{equation}\label{eqn:3.95*}
\hat{\nabla}_{X_1}Y_\alpha=0,\ \ 1\le\alpha\le n-1.
\end{equation}
Then, w.r.t $\{X_1,Y_\alpha\}$, \eqref{eqn:3.93} immediately follows
from the preceding conclusions.

Next, using \eqref{eqn:3.93}, we can apply \eqref{eqn:2.3} and
\eqref{eqn:2.10}, with $X=Y_2$ and $Y=Z=Y_1$, to obtain that
$c_2=(n+1)H/n<0$.

Finally, similar to the proof of $\hat{\nabla}X_1=0$ in Lemma
\ref{lm:3.6}, by \eqref{eqn:2.4} and \eqref{eqn:3.93}, we can show
that $\hat{\nabla}X_1=0$. From this, together with \eqref{eqn:3.93}
and \eqref{eqn:3.95*}, we can show by direct calculations that
$\hat{\nabla}K=0$.
\end{proof}


\noindent{\it The Completion of Theorem \ref{thm:1.2}'s Proof.}

Under the assumptions of Theorem \ref{thm:1.2}, we can apply Lemma
\ref{lm:3.11}, then as a direct consequence of Theorem 4.1 in
\cite{HLSV} we easily get the assertion.\qed

\vskip 3mm

\noindent{\it Proof of Corollaries.}

Let $x:M^n\rightarrow\mathbb{R}^{n+1}$, with $n=3$ (resp. $n=4$), be
a locally strongly convex affine hypersphere whose Ricci tensor is
parallel with respect to the Levi-Civita of the affine metric. Then,
by the classical de Rham-Wu's decomposition theorem \cite{W},
$(M^n,h)$ is locally isometric to a Riemannian product of Einstein
manifolds.

If $n=3$, then either $(M^3,h)$ is Einstein and thus $M^3$ is of
constant sectional curvature, or $(M^3,h)$ is locally isometric to a
Riemannian product $\mathbb{R}\times \tilde M^2$, where $\tilde M^2$
is a Riemannian manifold with constant sectional curvature. For both
of these cases, according to \cite{VLS} and Theorem \ref{thm:1.2},
we obtain Corollary \ref{cor:1.1}.

If $n=4$, then either $(M^4,h)$ is Einstein, or $(M^4,h)$ is locally
isometric to a Riemannian product $\mathbb{R}\times\tilde M^3$, or
$(M^4,h)$ is locally isometric to a Riemannian product $M_1^2\times
M_2^2$, where $\tilde M^3$, $M_1^2$ and $M_2^2$ are Riemannian
manifolds with constant sectional curvature. Then, for each of these
three cases, applying the results of \cite{HLV3}, Theorem
\ref{thm:1.1} and Theorem \ref{thm:1.2}, we obtain Corollary
\ref{cor:1.2}. \qed

\vskip 6mm


\vskip 6mm

\begin{flushleft}

Xiuxiu Cheng and Zejun Hu:\\
{\sc School of Mathematics and Statistics, Zhengzhou University,\\
Zhengzhou 450001, People's Republic of China;\\
Henan Key Laboratory of Financial Engineering, Zhengzhou 450001, People's Republic of China.}\\
E-mails: chengxiuxiu1988@163.com; huzj@zzu.edu.cn.

\vskip 2mm

Marilena Moruz:\\
{\sc KU Leuven, Department of Mathematics, Celestijnenlaan 200B -- Box 2400, BE-3001 Leuven, Belgium.}\\
E-mail: marilena.moruz@kuleuven.be.

\vskip 2mm
Luc Vrancken:

{\sc Univ. Valenciennes, EA4015-LAMAV, F-59313 Valenciennes, FRANCE;\\
 KU Leuven, Department of Mathematics, Celestijnenlaan 200B -- Box 2400, BE-3001 Leuven, Belgium.}\\
E-mails: luc.vrancken@univ-valenciennes.fr.

\end{flushleft}

\end{document}